\documentclass[11pt]{article}
\usepackage[letterpaper,margin=1in]{geometry}
\usepackage[T1]{fontenc}
\usepackage{lmodern}
\usepackage{microtype}
\usepackage[utf8]{inputenc}
\usepackage[english]{babel}
\usepackage{amsmath,amssymb,amsthm,mathtools}
\usepackage[hidelinks]{hyperref}
\numberwithin{equation}{section}

\newtheorem{theorem}{Theorem}[section]
\newtheorem{lemma}[theorem]{Lemma}
\newtheorem{proposition}[theorem]{Proposition}
\newtheorem{corollary}[theorem]{Corollary}

\theoremstyle{definition}
\newtheorem{definition}[theorem]{Definition}

\renewcommand{\P}{\mathrm{P}}
\newcommand{\E}{\mathrm{E}}
\newcommand{\1}{\mathbf{1}}
\newcommand{\unn}{\{1,\ldots,n\}}
\newcommand{\Om}{\Omega}

\newcommand{\setcoord}[3]{#1(#2\leftarrow #3)}
\DeclareMathOperator{\Inf}{Inf}

\DeclareMathOperator{\Piv}{Piv}

\begin{document}

\title{How much do increasing events occur disjointly?}

\author{
Pierre Tesio
\thanks{Universit\'e Paris-Saclay, CNRS, Laboratoire de math\'ematiques d'Orsay, 91405 Orsay.}
}
\date{}
\maketitle
\begin{abstract}
We establish two general quantitative lower bounds on the gap
in the classical BK inequality for increasing events.
The first one is an analogue of
Talagrand's quantitative FKG inequality \cite{Talagrand1996}, it is a function
 of the influences. The second one is specific
to the BK inequality, it involves both the influences and the probabilities of symmetric
disjoint pivots, introduced in \cite{CerfTesio2026}.
% These two results illustrate different mechanisms explaining 
% the gap in the BK inequality.
\end{abstract}
\section{Introduction}
We fix an integer \(n \geq 1\), a real number \(p\) in \(]0,1[\) and we set
\(q=1-p\).
We endow the hypercube \(\Om=\{0,1\}^n\) with the Bernoulli product measure 
of parameter p, denoted by \(\mu_p\). An element \(x=(x_1,\dots,x_n)\) of $\Om$ is smaller than or equal to
an element \(y=(y_1,\dots,y_n)\) of \(\Om\)
if \(x_i\leq y_i\) for \(1\leq i\leq n\).
A subset \(A\) of \(\Om\) is increasing if the map 
\(x\in\Om\mapsto 1_A(x)\) is non-decreasing.

The van den Berg-Kesten inequality \cite{vanDenBergKesten1985},
usually referred to as the BK inequality, 
is a fundamental tool in Bernoulli percolation.
For two increasing events
\(A\) and \(B\), it states that
\[
\mu_p(A\circ B)
\,\leq\,
\mu_p(A)\mu_p(B),
\]
The definition of the disjoint occurrence \(A\circ B\) rests on
the notion of witness.
\begin{definition}[Witness]
Let \(A\) be a subset of \(\Omega\) and let \(x\) belong to \(A\).
A subset \(I\) of \(\unn\) is a witness for \(A\) in \(x\) if any \(y\) in \(\Omega\)
which coincides with \(x\) on the components whose index is in \(I\)
is in \(A\).
\end{definition}

\noindent
We define \(A\circ B\) to be the event that
\(A\) and \(B\) can be witnessed by two disjoint sets of coordinates.
The purpose of this paper is to study a quantitative form of this inequality.
We try to understand the gap
\[
\Delta_{BK}(A,B)
\,=\,
\mu_p(A)\mu_p(B)-\mu_p(A\circ B).
\]
Our main inspiration for this work comes from Talagrand's quantitative version of the FKG
inequality. If \(A\) and \(B\) are increasing, then the FKG inequality states that
\[
\mu_p(A\cap B)
\,\geq\,
\mu_p(A)\mu_p(B)\,,
\]
meaning that increasing events are positively correlated.
In \cite{Talagrand1996}, Talagrand proved that this positive correlation
can be bounded from below
in terms of the influences of \(A\) and \(B\). 

\begin{definition}[Pivotality and influence]
For a configuration \(x\), a coordinate \(i\), and
\(\varepsilon\in\{0,1\}\), write \(x(i\leftarrow\varepsilon)\) for the
configuration obtained from \(x\) by setting coordinate \(i\) equal to
\(\varepsilon\). We say that \(i\) is pivotal for an event \(A\) in
\(x\) if
\[
\mathbf 1_A\bigl(x(i\leftarrow1)\bigr)
\,\neq\,
\mathbf 1_A\bigl(x(i\leftarrow0)\bigr).
\]
The influence of coordinate \(i\) on \(A\), denoted by
\(\Inf_i(A)\), is the probability that \(i\) is pivotal for \(A\). 
We also denote by \(\operatorname{Piv}(i,A)\) the event that coordinate \(i\) is pivotal for \(A\).
\end{definition}

\noindent
To discuss Talagrand's result, we specialize temporarily to the case
\(p=1/2\), and we write \(\mu_{1/2}\) for the uniform measure on \(\Om\). Let
\[
\varphi(x) 
\,=\,
\frac{x}{\log(\frac{e}{x})}
\] for $x$ in \(]0,1]\),
with \(\varphi(0)=0\). In \cite{Talagrand1996}, Talagrand proved that there exists a universal constant \(K>0\) such that,
for every \(n\geq1\) and all increasing events \(A\) and \(B\) in \(\Om=\{0,1\}^n\),
\[
\mu_{1/2}(A\cap B)
-
\mu_{1/2}(A)\mu_{1/2}(B)
\,\geq\,
\frac{1}{K}
\varphi\left(
\frac{1}{4}
\sum_{i=1}^n
\Inf_i(A)\Inf_i(B)
\right)\,.
\]
Thus, the FKG gap is controlled from below 
by a quantity measuring to which extent
the two events depend on the same coordinates. A first consequence
of our results is that the BK gap is controlled by the same quantity.
The BK gap also depends on the geometry of witnesses. This gives rise to
an additional bound, specific to the BK inequality. These results are 
obtained by applying Talagrand's method in the framework of the strengthened
BK inequality,
studied in \cite{CerfTesio2026}, that we recall next.
\begin{definition}[Disjoint testimony]
Let \(A\) and \(B\) be two subsets of \(\Omega\).
  The disjoint testimony of \(A\) and \(B\) is the subset
  \(A\boxtimes B\) of \(A\times B\) consisting of the pairs \((x,y)\)
  admitting disjoint witnesses for \(A\) and \(B\), i.e.,
  \begin{multline*}
    A\boxtimes B
    \,=\,
    \smash{\Big\{}\,(x,y)\in A\times B:\exists\,I,J\subset\unn,\quad
    I\cap J\,=\,\varnothing,
    I \text{ is a witness for } A \text{ in } x,\quad
    \\
    J \text{ is a witness for } B \text{ in } y
  \,\smash{\Big\}}\,.
  \end{multline*}
\end{definition}

\noindent
Denoting by \(\mu_p\otimes\mu_p\) the product measure on \(\Omega\times\Omega\), the
strengthened inequality reads
\[
\mu_p(A\circ B)
\, \leq \,
\mu_p\otimes\mu_p(A\boxtimes B)
.\]
% This inequality is implicit in the original work of van den Berg and Kesten
% \cite{vanDenBergKesten1985}. Its deduction from their original results
% was pointed out by Radhakrishnan and Tassion
% \cite{RadhakrishnanTassion2026}.
The strong BK inequality therefore 
yields the positive decomposition
\[
\Delta_{\mathrm{BK}}(A,B)
\, =\,
D(A,B)
+
\bigl(\mu_p\otimes\mu_p(A\boxtimes B)-\mu_p(A\circ B)\bigr),
\]
where
\[
D(A,B)=\mu_p(A)\mu_p(B)-\mu_p\otimes\mu_p(A\boxtimes B).
\]
The two terms correspond to different questions. The first measures the
failure to find disjoint witnesses in two independent configurations,
whereas the second measures what is lost when one passes from two
configurations admitting disjoint witnesses to disjoint occurrence in a
single configuration. The present paper focuses on the first part of the gap, 
which we refer to as the product space gap. The second part corresponds to a different 
quantitative problem. 
The quantity \(D(A,B)\) is the probability
of the pairs in \(A\times B\) which do not admit
disjoint witnesses. 
In \cite{CerfTesio2026}, we characterized the events such that \(D(A,B)=0\).
Here, we provide quantitative lower bounds on \(D(A,B)\).
Our first main result is the following analogue of
Talagrand's inequality for the BK
gap.
% In fact, we prove the stronger statement obtained by replacing the BK
% gap with \(D(A,B)\). 
Recall that \(\varphi(x) = x/\log(e/x)\) for \(x\) in \(]0,1]\), with \(\varphi(0)=0\).
Set \(p\wedge q \coloneqq \min\{p,q\}\). Define also
\[
B(p)
\,=\,
\frac{q-p}{2pq(\log q-\log p)}
\]
when \(p\neq1/2\), and \(B(1/2)=1\). This is the hypercontractivity
constant appearing in the work of Keller and Kindler
\cite{KellerKindler2013}. 
\begin{theorem}
\label{thm:influence-bound}
There exists a universal constant \(c>0\) such that, for every \(p\) in
\(]0,1[\), every \(n\geq1\), and all increasing events \(A\) and \(B\),
\[
\Delta_{\mathrm{BK}}(A,B)
\,\geq\,
c\,
\frac{(p\wedge q)^2}{pq\,B(p)}
\sqrt{\mu_p(A)\mu_p(B)}
\,
\varphi\left(
\dfrac{
pq\sum_{i=1}^n
\Inf_i(A)\Inf_i(B)
}{
\sqrt{\mu_p(A)\mu_p(B)}
}
\right)\,.
\]
The right-hand side is understood to be zero when
\(\mu_p(A)\mu_p(B)=0\).
\end{theorem}

\noindent
In the uniform case $p=1/2$, this bound becomes 
\[
\Delta_{\mathrm{BK}}(A,B)
\,\geq\,
c\,
\sqrt{\mu_{1/2}(A)\mu_{1/2}(B)}
\varphi\left(
    \frac{
        \sum_{i=1}^n
        \Inf_i(A)\Inf_i(B)
    }{
        4\sqrt{\mu_{1/2}(A)\mu_{1/2}(B)}
    }
\right)\,.
\]
We actually prove this result for the product space gap $D(A,B)$,
which is smaller than or equal to $\Delta_{\mathrm{BK}}(A,B)$.
This result may be viewed as a BK analogue
of Talagrand's quantitative FKG inequality, extended to the biased case.
For the FKG inequality, 
the corresponding biased extension can be deduced from
the work of Keller and Kindler \cite{KellerKindler2013}. In their study of
the quantitative relation between noise sensitivity and influences, they
gave a simplified proof of Talagrand's fundamental estimates on influences,
extended them to higher Fourier levels, biased product measures and non-monotone 
events.
The inequality of theorem~\ref{thm:influence-bound} also keeps track of the 
probabilities of $A$ and $B$ in the bound, which is a slight improvement
compared to the original estimate. This could easily be 
obtained in Talagrand's argument for the FKG case.

Influences do not capture the whole obstruction to disjoint
occurrence. Even when the influences
are small, every witness for \(A\) may be forced
to intersect every witness for \(B\). 
The disjoint testimony formulation
helps to identify
another quantity underlying such collisions. We recall the relevant
notion, the symmetric disjoint pivot, introduced in \cite{CerfTesio2026}.

\begin{definition}[Symmetric disjoint pivot]
A coordinate \(i\) is a symmetric disjoint pivot for \((A,B)\) in
\((x,y)\) if
\[
\bigl(x(i\leftarrow 0),y(i\leftarrow 0)\bigr)
\,\notin\,
A\boxtimes B\,,
\]
whereas
\[
\bigl(x(i\leftarrow 1),y(i\leftarrow 0)\bigr)
\in
A\boxtimes B
\quad\text{and}\quad
\bigl(x(i\leftarrow 0),y(i\leftarrow 1)\bigr)
\in
A\boxtimes B\,.
\]
\end{definition}

\noindent
In other words, closing coordinate \(i\) in both configurations prevents
disjoint testimony, whereas opening it in either configuration restores it.
We denote by \(\pi_i(A,B)\) the probability that \(i\) is a symmetric
disjoint pivot for \((A,B)\).
The same object appears in the analytic
framework of Radhakrishnan and Tassion
\cite{RadhakrishnanTassion2026}: in their notation, the operator
\(D_iF\) is the indicator of the symmetric disjoint pivot event.

The symmetric disjoint pivot probability and influences appear naturally 
in two key identities, which form the starting point of our proofs.
Accordingly, define \(\rho(A,B)=(\rho_i(A,B))_{1\leq i\leq n}\) by
\[
\forall i \in \{1,\dots,n\}
\qquad \rho_i(A,B)
\, = \,
\Inf_i(A)\Inf_i(B)+\pi_i(A,B),
\]
and its \(\ell^2\)-norm by
\[
\|\rho(A,B)\|_2
\, =\,
\left(\sum_{i=1}^n \rho_i(A,B)^2\right)^{1/2}
.\]
Our second main result bounds the gap in terms of this
quantity. 
\begin{theorem}
\label{thm:SDP-bound}

For every \(p\in]0,1[\), every \(n\geq 1\), and all increasing events
\(A\) and \(B\),
\[
\Delta_{BK}(A,B)
\,\geq\,
\frac{
    (p\wedge q)^2 \|\rho(A,B)\|_2
}{
    12
    \sqrt{
        \log\Biggl(
            \dfrac{
                6\sqrt{\mu_p(A)\mu_p(B)}
            }{
                (p\wedge q)^2\|\rho(A,B)\|_2
            }
        \Biggr)
    }
}
\,.
\]
The right-hand side is understood to be zero when
\(\|\rho(A,B)\|_2=0\).
\end{theorem}

\noindent
Theorem~\ref{thm:influence-bound} and
theorem~\ref{thm:SDP-bound} capture two different
obstructions to equality. The first involves the influences and has
the Talagrand logarithmic scale, whereas the second controls the
gap with the \(\ell^2\)-norm of influences and symmetric disjoint pivot probabilities.
We illustrate the behavior of these bounds with classical examples. 
Let $n$ be a positive integer, and take $k$ between $1$ and $n$. 
The symmetric threshold associated 
to \((k,n)\) is the event that the
configuration contains at least \(k\) open coordinates. More precisely,
\[
H_{n,k}
\,=\,
\biggl\{\,
x\in\{0,1\}^n:
\sum_{i=1}^n x_i
\,\geq\,
k
\,
\biggr\}
\,.
\]
For the threshold event $H_{n,1}$, which is the OR event, we prove in section 
\ref{sec:examples} that 
the bound of theorem \ref{thm:SDP-bound} is sharp for the product space gap $D(A,B)$. 
More precisely, for two quantities $f$ and $g$ depending on $n$, we 
write $f\asymp g$ 
if there exists constants $0<c<C<\infty$,  independent of n, such that
$cg<f<Cg$ for all $n$ large enough. When these constants are allowed to depend on
$p$, we write $f\asymp_p g$.
We have 
\[
\frac{
\|\rho(H_{n,1},H_{n,1})\|_2
}{
\sqrt{
\log\left(
\dfrac{
6
\mu_p(H_{n,1})
}{
\min(p,q)^2 \|\rho(H_{n,1},H_{n,1})\|_2
}
\right)
}
}
\,\asymp_p\,
nq^{2n}
\,\asymp_p\,
D(H_{n,1},H_{n,1})
\,.
\]
For $n=2m+1$, consider the threshold event $H_n=H_{2m+1,m+1}$, which is the majority event, we prove that,
when considering the uniform measure $p=1/2$, the bound of theorem \ref{thm:influence-bound} is of correct order :
\[
\mu(H_n)\,
\varphi\left(
\frac{\frac14 \sum_{i=1}^n
\Inf_i(H_n)\Inf_i(H_n)
 }{\mu(H_n)}
\right)
\asymp 1
\asymp
D(H_n,H_n).
\]
Most subsequent developments of the BK inequality have concerned extensions
of its range of validity rather than quantitative improvements.
Notably, Reimer extended the inequality to arbitrary events in
\cite{Reimer2000}. Talagrand gave a direct proof of the BK inequality by induction
on the dimension, based on a decomposition into coordinate sections in 
\cite{TalagrandBK}. One of his results, applied to \(A\boxtimes B\), also yields
the strengthened BK 
inequality, and his section decomposition technique is one of the key
elements of our proof. 

Other quantitative results arise in geometric
applications to critical planar percolation. Beffara and Nolin obtained a polynomial
improvement of the comparison given by Reimer's lemma for monochromatic
and polychromatic arm events in \cite{BeffaraNolin2011}. 
More recently, 
Radhakrishnan and Tassion derived  quantitative
versions of the FKG inequality and Reimer's lemma, and used them to prove
strict inequalities between arm exponents in
\cite{RadhakrishnanTassion2026}. 
These quantitative estimates are quite different from what we present here. 
They are in fact a continuous version of the formula we described in 
\cite{CerfTesio2026}, in the case of Reimer's lemma. 
The quantitative improvement that they obtain then follows 
from a careful analysis of the geometry of the events under consideration.

% These estimates are adapted to the
% geometry of the particular percolation events under consideration. To the
% best of our knowledge, no general lower bound on
% \(\mu_p(A)\mu_p(B)-\mu_p(A\circ B)\) was previously available.

The paper is organized as follows. In section~\ref{sec:sketch}, we provide a sketch
of the proofs of both theorems. In section~\ref{sec:exact-section-identities},
we establish two identities which will form the starting point for 
both proofs. 
Section~\ref{sec:analytic-estimates} contains
key analytic estimates inspired by the work of Talagrand \cite{Talagrand1996}, Keller and Kindler \cite{KellerKindler2013}.
In section~\ref{sec:proofs}, we prove the two main results.
Finally, section~\ref{sec:examples} contains examples illustrating the
sharpness and the limitations of the bounds. 
\section{Proof strategy}
\label{sec:sketch}
We start by giving an overview of the proof of 
theorem~\ref{thm:influence-bound}. For simplicity,
we only discuss here the uniform case \(p=1/2\) and denote the uniform measure 
on \(\{0,1\}^n\) by \(\mu=\mu_{1/2}\).
We follow closely the strategy of Talagrand in \cite{Talagrand1996}. 

To understand the structure of the product space gap \(D(A,B)\), the first step is to decompose it by
conditioning on the value of the last coordinate. 
Let $A$ and $B$ be two increasing events on the hypercube. We define, for
\(\varepsilon\) in \(\{0,1\}\), the sections along the last
coordinate as
\[
A_\varepsilon
=
\left\{
x\in\{0,1\}^{n-1}:
(x,\varepsilon)\in A
\right\},
\qquad
B_\varepsilon
=
\left\{
x\in\{0,1\}^{n-1}:
(x,\varepsilon)\in B
\right\}.
\]
We also extend this notion to the product space: for
an event \(C\) on the product cube \(\{0,1\}^n \times \{0,1\}^n\)
and \(\varepsilon,\eta\) in \(\{0,1\}\), we 
define the sections along the last coordinate by 
\[
C_{\varepsilon \eta}
=
\left\{
(x,y)\in\{0,1\}^{n-1}\times\{0,1\}^{n-1}:
\bigl((x,\varepsilon),(y,\eta)
\bigr)\in C
\right\}\,.
\]
We still write \(\mu\) for the uniform measure on lower dimensional 
hypercubes.
Because the measure is product, we can decompose an event along a given 
coordinate, by conditioning on its value. 
For instance, 
\[
\mu(A)=\frac12 \mu(A_0)+\frac12 \mu(A_1).
\]
We denote the product measure $\mu\otimes\mu$ as $\nu$. Decomposing the gap along the last coordinate therefore gives 
\[
\begin{aligned}
\mu(A)\mu(B)-\nu(A\boxtimes B)
=&{}
\frac14\Bigl[
\mu(A_0)\mu(B_0)
-
\nu\bigl((A\boxtimes B)_{00}\bigr)
\Bigr]
+
\frac14\Bigl[
\mu(A_1)\mu(B_0)
-
\nu\bigl((A\boxtimes B)_{10}\bigr)
\Bigr]
\\
&+
\frac14\Bigl[
\mu(A_0)\mu(B_1)
-
\nu\bigl((A\boxtimes B)_{01}\bigr)
\Bigr]
+
\frac14\Bigl[
\mu(A_1)\mu(B_1)
-
\nu\bigl((A\boxtimes B)_{11}\bigr)
\Bigr].
\end{aligned}
\]
A careful analysis of the four sections of the disjoint testimony, along with
elementary computations using the monotonicity of \(A\) and \(B\) gives
the exact decomposition
\begin{equation}
\label{eq:section-decomposition-uniform}
D(A,B)
\,=\,
\frac12 D(A_1,B_0)
+
\frac12 D(A_0,B_1)
+
\frac14
\Bigl(
\Inf_n(A)\Inf_n(B)+\pi_n(A,B)
\Bigr)
\, ,
\end{equation}
where we recall that \(\pi_n(A,B)\) is the probability that the last coordinate
is a symmetric disjoint pivot for \((A,B)\).
This is the central identity and the starting point to apply
Talagrand's method.
For instance, Talagrand obtains in \cite{Talagrand1996} for the FKG inequality, 
with obvious notations:
\[
\Delta_{FKG}(A,B)
\,=\,
\frac12\Delta_{FKG}(A_0,B_0)
+
\frac12\Delta_{FKG}(A_1,B_1)
+ 
\frac14
\Bigl(
\Inf_n(A)\Inf_n(B)
\Bigr)\,.
\]
We only focus on the proof of theorem~\ref{thm:influence-bound} for now, so
we can
ignore the last term in equation \eqref{eq:section-decomposition-uniform} and
consider the inequality 
\begin{equation}
\label{eq:inequality-section-decomposition-uniform}
D(A,B)
\,\geq\,
\frac12 D(A_1,B_0)
+
\frac12 D(A_0,B_1)
+
\frac14
\Inf_n(A)\Inf_n(B).
\end{equation}
At this point, it is tempting to iterate the argument on the two lower-dimensional gaps to recover
the influences of all coordinates. This, however, does not quite work.
On one hand, iteratively decomposing the lower-dimensional gaps
\(D(A_1,B_0)\) and \(D(A_0,B_1)\) in \eqref{eq:inequality-section-decomposition-uniform} along coordinate \(n-1\) produces
the contribution
\[
\frac18
\Bigl(
\Inf_{n-1}(A_1)\Inf_{n-1}(B_0)+\Inf_{n-1}(A_0)\Inf_{n-1}(B_1)
\Bigr).
\]
On the other hand, pivotality with respect to coordinate \(n-1\) can also be
conditioned on the value of the last coordinate, so that
\[
\Inf_{n-1}(A)
\,=\,
\frac{\Inf_{n-1}(A_0)+\Inf_{n-1}(A_1)}{2},
\]
and the same identity goes for \(B\).
The contribution that one would like to recover is therefore
\[
\begin{aligned}
\frac14\,
\Inf_{n-1}(A)\Inf_{n-1}(B)
&=
\frac1{16}
\Bigl(
\Inf_{n-1}(A_0)+\Inf_{n-1}(A_1)
\Bigr)
\Bigl(
\Inf_{n-1}(B_0)+\Inf_{n-1}(B_1)
\Bigr).
\end{aligned}
\]
The difference between the desired contribution and the one produced by
the second section decomposition is therefore
\begin{equation}
\label{eq:section-decomposition-error-uniform}
\begin{aligned}
&\frac14
\Inf_{n-1}(A)\Inf_{n-1}(B)
-\frac18
\Bigl(
\Inf_{n-1}(A_1)\Inf_{n-1}(B_0)+\Inf_{n-1}(A_0)\Inf_{n-1}(B_1)
\Bigr)
\\
&=
\frac1{16}
\bigl(\Inf_{n-1}(A_1)-\Inf_{n-1}(A_0)\bigr)
\bigl(\Inf_{n-1}(B_1)-\Inf_{n-1}(B_0)\bigr).
\end{aligned}
\end{equation}
More generally, let
\[
\mathcal I(A,B)
\,=\,
\sum_{i=1}^n
\Inf_i(A)\Inf_i(B).
\]
% and define the influence averaged over the two sections appearing
% in the decomposition by
% \[
% \overline{\mathcal I}_n
% \,=\,
% \frac12\mathcal I(A_1,B_0)
% +
% \frac12\mathcal I(A_0,B_1).
% \]
For every \(i<n\),
\[
\Inf_i(A)
\,=\,
\frac12
\Bigl(
\Inf_i(A_0)+\Inf_i(A_1)
\Bigr),
\qquad
\Inf_i(B)
\,=\,
\frac12
\Bigl(
\Inf_i(B_0)+\Inf_i(B_1)
\Bigr).
\]
The same argument as in \eqref{eq:section-decomposition-error-uniform} gives
\begin{equation}
\label{eq:uniform-influences-decomposition}
\begin{aligned}
\mathcal I(A,B)
\,=\,&{}
\frac12\mathcal I(A_1,B_0)
+
\frac12\mathcal I(A_0,B_1)
+
\Inf_n(A)\Inf_n(B)
\\
&+
\frac14
\sum_{i<n}
\Bigl(
\Inf_i(A_1)-\Inf_i(A_0)
\Bigr)
\Bigl(
\Inf_i(B_1)-\Inf_i(B_0)
\Bigr).
\end{aligned}
\end{equation}
The last term is the error which prevents the naive induction from closing.
There is a useful interpretation of this error, which Talagrand
used in his original argument \cite{Talagrand1996}.
Let us fix an increasing event $A$ and a coordinate $i\in\{1,\ldots,n\}$. 
We define
\[
r_i(x)=2x_i-1,
\]
% For an increasing event \(A\), the definition of pivotality 
% directly gives the identity:
The following identity
is a direct consequence of the definition of pivotality:
\begin{equation}
\label{eq:pivotality-identity-1}
\mathbf{1}_{\Piv(i,A)}
\,=\,
\mathbf{1}_{A_1^{(i)}}
-
\mathbf{1}_{A_0^{(i)}}    
\end{equation}
where \(A_\varepsilon^{(i)}\) denotes the 
section of \(A\) obtained by fixing
coordinate \(i\) to \(\varepsilon\).
By conditioning on the value of the coordinate \(i\), we have
\begin{equation}
\label{eq:pivotality-identity-2}
\int_A r_i\,d\mu
\,=\, \frac12\Big(\mu\big(A_1^{(i)}\big)-\mu\big(A_0^{(i)}\big)\Big)
\,=\, \frac12\operatorname{Inf}_i(A).
\end{equation}
Furthermore, \eqref{eq:pivotality-identity-1} 
and \eqref{eq:pivotality-identity-2} give that, for every \(k \neq i\),
\begin{equation}
\begin{aligned}
\label{eq:harmonic-identity}
\int_{\Piv(k,A)}
r_i\,d\mu
\,&=\,
\int_{A_1^{(k)}} r_i\,d\mu
-
\int_{A_0^{(k)}} r_i\,d\mu
\\
&=\,
\frac12
\Bigl(
\Inf_i(A_1^{(k)})
-
\Inf_i(A_0^{(k)})
\Bigr).
\end{aligned}
\end{equation}
% Consequently, if we set
% \[
% \Gamma_k(A,B)
% \,=\,
% \sum_{i\neq k}
% \Bigl(
% \int_{\Piv(k,A)}
% r_i\,d\mu
% \Bigr)
% \Bigl(
% \int_{\Piv(k,B)}
% r_i\,d\mu
% \Bigr),
% \]
% then \eqref{eq:uniform-influences-decomposition} becomes, for \(k=n\),
With this interpretation, the identity 
\eqref{eq:uniform-influences-decomposition} becomes
\[
\mathcal I(A,B)
\,=\,
\frac12\mathcal I(A_1,B_0)
+
\frac12\mathcal I(A_0,B_1)
+
\Inf_n(A)\Inf_n(B)
+
\sum_{i<n}
\Bigl(
\int_{\Piv(n,A)}
r_i\,d\mu
\Bigr)
\Bigl(
\int_{\Piv(n,B)}
r_i\,d\mu
\Bigr),
\]
The crucial point is that these errors cannot be large for every coordinate
at the same time. More precisely,
Talagrand's decoupled estimate has, up to a universal constant $c$, the form
\[
\sum_{k=1}^n
\sum_{i\neq k}
\left|
\Bigl(
\int_{\Piv(k,A)}
r_i\,d\mu
\Bigr)
\Bigl(
\int_{\Piv(k,B)}
r_i\,d\mu
\Bigr)
\right|
\,\leq\,
c \,
\mathcal I(A,B)
\log\left(
\frac{c}{\mathcal I(A,B)}
\right).
\]
Since
\[
\mathcal I(A,B)
\,=\,
\sum_{k=1}^n
\Inf_k(A)\Inf_k(B)\,,
\]
there is a coordinate \(k\) for which the corresponding
error is controlled by
\(\Inf_k(A)\Inf_k(B)\), up to the logarithmic factor above.
After relabeling this coordinate as \(n\), the term
\(\frac14\Inf_n(A)\Inf_n(B)\)
in \eqref{eq:section-decomposition-uniform} can be used to compensate for the error 
terms in \eqref{eq:uniform-influences-decomposition}.
By reasoning by induction, we will prove that
Talagrand's decoupled estimate is precisely what is needed to control
the error and close the argument.
% The logarithmic loss in the decoupled estimate justifies
% the apparition of the function
% \[
% \varphi(u)
% \,=\,
% \frac{u}{\log(e/u)}
% \]
% in theorem~\ref{thm:influence-bound}.

We now turn to theorem~\ref{thm:SDP-bound}. One can prove
this result by a similar induction by using analogous intermediate 
estimates for the symmetric disjoint pivots, however, there is a much
shorter argument.
Let
\begin{equation}
\label{eq:defect-event}
\mathcal E(A,B)
=
(A\times B)\setminus(A\boxtimes B)
\end{equation}
be the defect event, so that
\[
\nu\bigl(\mathcal E(A,B)\bigr)
=
D(A,B).
\]
For every coordinate \(i\), define
\[
\chi_i(x,y)
=
(2x_i-1)(2y_i-1).
\]
The variables \(\chi_1,\ldots,\chi_n\) are independent centered
Rademacher variables under \(\nu\). The identity \eqref{eq:section-decomposition-uniform}
also has a Fourier interpretation. 
Indeed, we will prove that, by conditioning on the \(i\)-th coordinates of the two configurations,
we have
\begin{equation}
\label{eq:uniform-fourier-identity}
\int_{\mathcal E(A,B)}
\chi_i\,d\nu
=
\frac14
\Bigl(
\Inf_i(A)\Inf_i(B)
+
\pi_i(A,B)
\Bigr)
\,=\,
\frac14\rho_i(A,B).
\end{equation}
and therefore
\[
\left(
\sum_{i=1}^n
\left(
\int_{\mathcal E(A,B)}
\chi_i\,d\nu
\right)^2
\right)^{1/2}
=
\frac14\|\rho(A,B)\|_2.
\]
The point is that the left-hand side is simply the Euclidean norm of a
family of first-level Fourier coefficients of an indicator function.
An extension of proposition 2.2 in \cite{Talagrand1996} to
the product cube, gives that, for
an arbitrary event \(E\)
\[
\left(
\sum_{i=1}^n
\left(
\int_E \chi_i\,d\nu
\right)^2
\right)^{1/2}
\leq
\,
C\,\nu(E)
\sqrt{
\log\left(
\frac{2}{\nu(E)}
\right)
}
\,,
\]
for a universal constant \(C>0\).
Applying this estimate to
\(E=\mathcal E(A,B)\) then yields
\[
\frac14\|\rho(A,B)\|_2
\leq
C\,D(A,B)
\sqrt{
\log\left(
\frac{2}{D(A,B)}
\right)
}.
\]
This way, the whole vector of influence and symmetric-disjoint-pivot
contributions is controlled at once, without any induction.
Rewriting this inequality will give
\[
D(A,B)
\geq
c\,
\frac{
\|\rho(A,B)\|_2
}{
\sqrt{
\log\left(
\dfrac{
C\sqrt{\mu(A)\mu(B)}
}{
\|\rho(A,B)\|_2
}
\right)
}
},
\]
for universal constants \(c,C>0\), which is the uniform version of
theorem~\ref{thm:SDP-bound}.

\section{Exact section identities}
\label{sec:exact-section-identities}
In this section, we establish two key identities which
form the starting points of both proofs. These identities
naturally present the influences and symmetric disjoint pivot probabilities 
as the relevant quantities to understand the gap to equality.
The first identity is an exact decomposition of the product
space gap \(D(A,B)\) along the last coordinate. 
The second identifies certain first-level Fourier 
coefficients of the defect event $\mathcal E(A,B)$ defined in 
the previous section. 
Of course, the last 
coordinate does not play a special role, and all the results of this section
hold for any coordinate \(i\) in \(\unn\).

When events
are defined on a lower-dimensional cube, the same symbols \(\mu_p\) and
\(\nu_p\) denote the corresponding Bernoulli product measures.
We start by identifying the four sections along the last coordinate
of the disjoint testimony.

\begin{lemma}
\label{lem:sections-of-disjoint-testimony}

Let \(A\) and \(B\) be increasing events on the discrete cube. The
four sections of the disjoint testimony \(A\boxtimes B\) along the
last coordinate are given by
\[
(A\boxtimes B)_{00}
=
A_0\boxtimes B_0,
\qquad
(A\boxtimes B)_{10}
=
A_1\boxtimes B_0,
\qquad
(A\boxtimes B)_{01}
=
A_0\boxtimes B_1,
\]
and
\[
(A\boxtimes B)_{11}
=
(A_1\boxtimes B_0)
\cup
(A_0\boxtimes B_1).
\]
\end{lemma}

\begin{proof}
We recall that, for increasing events \(A\) and \(B\),
\(A_0 \subseteq A_1\,\) and \(B_0 \subseteq B_1\).
Let $x,y$ belong to \(\{0,1\}^{n-1}\). We only prove the identity for the
$11$ section, 
as the other cases are 
similar and easier. 

Suppose that \((x,y)\) is in \((A\boxtimes B)_{11}\),
and let \(K,L\) be disjoint witnesses for \(A\) in \((x,1)\) and \(B\)
in \((y,1)\). Since \(K\cap L=\varnothing\), the last coordinate belongs
to at most one of \(K\) and \(L\). If \(n\) does not belong to \(L\), then
\(K\setminus\{n\}\) witnesses \(A_1\), while \(L\) witnesses \(B_0\),
and therefore
\((x,y)\in A_1\boxtimes B_0\).
If \(n\) does not belong to \(K\), the symmetric argument gives
\((x,y)\in A_0\boxtimes B_1\).
At least one of these two cases always occurs, and hence
\begin{equation}
\label{eq:section-11-inclusion}
(A\boxtimes B)_{11}
\subseteq
(A_1\boxtimes B_0)\cup(A_0\boxtimes B_1).
\end{equation}
Conversely, suppose that \((x,y)\) is in \((A_1\boxtimes B_0)\cup(A_0\boxtimes B_1)\). 
If
\((x,y)\in A_1\boxtimes B_0\), take disjoint witnesses \(K,L\subseteq\{1,\dots,n-1\}\) for \(A_1\) and \(B_0\).
Then \(K\cup\{n\}\) witnesses \(A\) in \((x,1)\), while \(L\) witnesses
\(B\) in \((y,1)\). Thus
\((x,y)\) is in \((A\boxtimes B)_{11}\).
The same argument applies to \(A_0\boxtimes B_1\), which proves
the converse inclusion to \eqref{eq:section-11-inclusion} 
and completes the proof.
\end{proof}
\noindent
We are now ready to state the main result of this section, which expresses
the decomposition of the product space gap along the last coordinate. 

\begin{proposition}
\label{prop:exact-section-identity}
Let A and B be increasing events on the discrete cube. Then 
\[
\begin{aligned}
D(A,B) 
\,=\,
(q-p)_+ D(A_0,B_0)
+
(p\wedge q) D(A_1,B_0)
+
(p\wedge q) D(A_0,B_1)
\\
+
(p-q)_+ D(A_1,B_1)
+
(p \wedge q)^2\Bigl(
\Inf_n(A)\Inf_n(B)
+ 
\pi_n(A,B)
\Bigr)
\\
+
(p-q)_+
\mu_p \otimes \mu_p\left(
(A_1\boxtimes B_1)
\setminus
\Bigl(
(A_1\boxtimes B_0)
\cup
(A_0\boxtimes B_1)
\Bigr)
\right)
\end{aligned}
\]
\end{proposition}
\begin{proof}
For simplicity, we write \(\nu_p\) for the product measure \(\mu_p\otimes\mu_p\).
We start by decomposing the gap along the value of the last coordinate, and use 
lemma~\ref{lem:sections-of-disjoint-testimony} to identify the 
four sections of the disjoint testimony.
This gives
\begin{equation}
\begin{aligned}
\label{fdcpo}
&D(A,B)
={} q^2 \Bigl[
    \mu_p(A_0)\mu_p(B_0)
    - \nu_p(A_0 \boxtimes B_0)
\Bigr]+ pq \Bigl[
    \mu_p(A_1)\mu_p(B_0)
    - \nu_p(A_1 \boxtimes B_0)
\Bigr] \\
&+ pq \Bigl[
    \mu_p(A_0)\mu_p(B_1)
    - \nu_p(A_0 \boxtimes B_1)
\Bigr] 
 + p^2 \Bigl[
    \mu_p(A_1)\mu_p(B_1)
    - \nu_p\Bigl(
            (A_1 \boxtimes B_0)
            \cup (A_0 \boxtimes B_1)
    \Bigr)
\Bigr].
\end{aligned}
\end{equation}
We only need to rework the first and last terms to prove
the proposition. For simplicity,
we write,  
\[
\begin{aligned}
\Delta
&\coloneqq
D(A,B)
-
pq\Bigl(
    D(A_1,B_0)+D(A_0,B_1)
\Bigr)
\\
&=
q^2\Bigl[
    \mu_p(A_0)\mu_p(B_0)
    -
    \nu_p(A_0\boxtimes B_0)
\Bigr]
+
p^2\Bigl[
    \mu_p(A_1)\mu_p(B_1)
    -
    \nu_p\Bigl(
        (A_1\boxtimes B_0)
        \cup
        (A_0\boxtimes B_1)
    \Bigr)
\Bigr].
\end{aligned}
\]
The first step is to combine the two terms with the help of the following
identities: 
\[
p^2=(p\wedge q)^2+(p-q)_+ \quad 
q^2=(p\wedge q)^2+(q-p)_+.
\]
We obtain
\begin{equation}
\label{eq:Delta_1}
\begin{aligned}
\Delta
={}&
(p\wedge q)^2
\Bigl[
    \mu_p(A_0)\mu_p(B_0)
    +
    \mu_p(A_1)\mu_p(B_1)
    -
    \nu_p(A_0\boxtimes B_0)
    -
    \nu_p\Bigl(
        (A_1\boxtimes B_0)
        \cup
        (A_0\boxtimes B_1)
    \Bigr)
\Bigr]
\\
&+
(q-p)_+
\Bigl[
    \mu_p(A_0)\mu_p(B_0)
    -
    \nu_p(A_0\boxtimes B_0)
\Bigr]
\\
&+
(p-q)_+
\Bigl[
    \mu_p(A_1)\mu_p(B_1)
    -
    \nu_p\Bigl(
        (A_1\boxtimes B_0)
        \cup
        (A_0\boxtimes B_1)
    \Bigr)
\Bigr].
\end{aligned}
\end{equation}
The second step is to develop the union in the first summand of \eqref{eq:Delta_1}:
\begin{equation}
    \label{eq:PDS expansion}
\nu_p\Bigl(
    (A_1\boxtimes B_0)
    \cup
    (A_0\boxtimes B_1)
\Bigr)
=
\nu_p(A_1\boxtimes B_0)
+
\nu_p(A_0\boxtimes B_1)
-
\nu_p\Bigl(
    (A_1\boxtimes B_0)
    \cap
    (A_0\boxtimes B_1)
\Bigr).
\end{equation}
\noindent
We recognize that the event 
\[
(A_1\boxtimes B_0)
\cap
(A_0 \boxtimes B_1)
\setminus
A_0 \boxtimes B_0
\]
is, by definition, the event that the last coordinate is a symmetric disjoint 
pivot for the pair $(A,B)$. Its probability is therefore $\pi_n(A,B)$.
To make $D(A_1,B_0)$ and $D(A_0,B_1)$ appear, we
add and subtract 
\[
(p\wedge q)^2
\bigl(
\mu_p(A_1)\mu_p(B_0)+\mu_p(A_0)\mu_p(B_1)
\bigr)
\] to 
\eqref{eq:Delta_1}. By regrouping the terms, \eqref{eq:Delta_1} now becomes 
\begin{equation}
\label{eq:Delta_2}
\begin{aligned}
\Delta
={}&
(p\wedge q)^2
\Biggl[
\pi_n(A,B)
+
\mu_p(A_0)\mu_p(B_0)
+
\mu_p(A_1)\mu_p(B_1)
-
\mu_p(A_1)\mu_p(B_0)
-
\mu_p(A_0)\mu_p(B_1)
\Biggr]
\\
&
+
(p\wedge q)^2
\Bigl[
    D(A_1,B_0)+D(A_0,B_1)
\Bigr]
+
(q-p)_+D(A_0,B_0)
\\
&+
(p-q)_+
\Bigl[
    \mu_p(A_1)\mu_p(B_1)
    -
    \nu_p\Bigl(
        (A_1\boxtimes B_0)
        \cup
        (A_0\boxtimes B_1)
    \Bigr)
\Bigr].
\end{aligned}
\end{equation}
\noindent
The terms involving $\mu_p$ in the first summand of \eqref{eq:Delta_2}
factor as the product of the influences:
\begin{equation}
\label{eq:influence-factorization}
\begin{aligned}
&\mu_p(A_0)\mu_p(B_0)
+\mu_p(A_1)\mu_p(B_1)
-\mu_p(A_1)\mu_p(B_0)
-\mu_p(A_0)\mu_p(B_1)
\\
&\qquad=
\bigl(\mu_p(A_1)-\mu_p(A_0)\bigr)
\bigl(\mu_p(B_1)-\mu_p(B_0)\bigr)
\,=\,
\operatorname{Inf}_n(A)\operatorname{Inf}_n(B).
\end{aligned}
\end{equation}
By definition of \(\Delta\) and the use of the identity 
\(p\wedge q = pq+(p\wedge q)^2\), \eqref{eq:Delta_2} becomes
\begin{equation}
\label{eq:D-recursion}
\begin{aligned}
D(A,B)
={}&
(p\wedge q)^2
\Bigl[
    \pi_n(A,B)
    +
    \operatorname{Inf}_n(A)\operatorname{Inf}_n(B)
\Bigr]
+
(p\wedge q)
\Bigl[
    D(A_1,B_0)
    +
    D(A_0,B_1)
\Bigr]
\\
&+
(q-p)_+D(A_0,B_0)
+
(p-q)_+
\Bigl[
    \mu_p(A_1)\mu_p(B_1)
    -
    \nu_p\Bigl(
        (A_1\boxtimes B_0)
        \cup
        (A_0\boxtimes B_1)
    \Bigr)
\Bigr].
\end{aligned}
\end{equation}
We now simply need to add and subtract \((p-q)_+\,\nu_p(A_1\boxtimes B_1)\)
to \eqref{eq:D-recursion}. 
Since $A$ and $B$ are increasing, we have
\[
(A_1\boxtimes B_0)\cup(A_0\boxtimes B_1)\subseteq A_1\boxtimes B_1\,.
\]
This concludes the proof of the proposition.
\end{proof}
\noindent
In the uniform case, the coefficients \((p-q)_+\) and
\((q-p)_+\) vanish. Proposition~\ref{prop:exact-section-identity}
reduces to the identity \eqref{eq:section-decomposition-uniform}:
\[
D(A,B)
\,=\,
\frac{1}{2}D(A_1,B_0)
+
\frac{1}{2}D(A_0,B_1)
+
\frac{1}{4}\rho_n(A,B)\,.
\]
We present next an identity that relates the Fourier coefficients 
of the indicator function of the event 
\(\mathcal E(A,B)=(A\times B)\setminus(A\boxtimes B)\)
to the influences and symmetric disjoint pivot probabilities. 
For \(x,y\in\{0,1\}^n\) and \(i \in \unn\), define 
\begin{equation}
\label{eq:chi_i}
r_i(x)
\,=\,
\frac{x_i-p}{\sqrt{pq}}\,,
\qquad
\chi_i(x,y)
\,=\,
r_i(x)r_i(y)\,.
\end{equation}
The functions \(\chi_1,\ldots,\chi_n\) form an orthonormal family in
\(L^2(\nu_p)\).

\begin{proposition}
\label{prop:diagonal-fourier}
For every coordinate \(i\) in \(\{1,\dots,n\}\) and increasing events 
\(A\) and \(B\), we have
\[
\int_{\mathcal E(A,B)}
\chi_i\,d\nu_p
\,=\,
pq\,\Bigl(
    \Inf_i(A)\Inf_i(B)
    +
    \pi_i(A,B)
\Bigr).
\]
\end{proposition}

\begin{proof}
Without loss of generality, we consider the last coordinate.
The function \(\chi_n\) takes the values
\[
\chi_n(0,0)
\,=\,
\frac{p}{q}\,,
\qquad
\chi_n(1,0)
\,=\,
\chi_n(0,1)
\,=\,
-1\,,
\qquad
\chi_n(1,1)
\,=\,
\frac{q}{p}\,.
\]
Conditioning on the last coordinates of the two configurations
gives

\[
\begin{aligned}
\int_{\mathcal E(A,B)}
\chi_n\,d\nu_p
\,={}&
q^2\frac{p}{q}
    \nu_p\Bigl(\bigl(\mathcal E(A,B)\bigr)_{00}\Bigr)
-pq \,
    \nu_p\Bigl(\bigl(\mathcal E(A,B)\bigr)_{10}\Bigr)
\\
&-pq\,\nu_p\Bigl(\bigl(\mathcal E(A,B)\bigr)_{01}\Bigr)
+
p^2\frac{q}{p}
    \nu_p\Bigl(\bigl(\mathcal E(A,B)\bigr)_{11}\Bigr).
\end{aligned}
\]
After developing each term, we obtain
\begin{equation}
\label{eq:diagonal-fourier-1}
\begin{aligned}
\int_{\mathcal E(A,B)}
\chi_n\,d\nu_p
\,={}&
pq\Bigl[
\mu_p(A_0)\mu_p(B_0)
-\nu_p(A_0\boxtimes B_0)
-
\mu_p(A_1)\mu_p(B_0)
+
\nu_p(A_1\boxtimes B_0)
-
\mu_p(A_0)\mu_p(B_1)
\\
&
+
\nu_p(A_0\boxtimes B_1)
+
\mu_p(A_1)\mu_p(B_1)
-\nu_p\bigl((A_0\boxtimes B_1)\cup(A_1\boxtimes B_0)\bigr)
\Bigr].
\end{aligned}    
\end{equation}
As in \eqref{eq:influence-factorization}, the terms involving $\mu_p$ factor as the 
product of the influences:
\[
\begin{aligned}
&\mu_p(A_0)\mu_p(B_0)
+\mu_p(A_1)\mu_p(B_1)
-\mu_p(A_1)\mu_p(B_0)
-\mu_p(A_0)\mu_p(B_1)
\\
&\qquad=
\operatorname{Inf}_n(A)\operatorname{Inf}_n(B).
\end{aligned}
\]
After expanding the union in the last term of \eqref{eq:diagonal-fourier-1} 
as in \eqref{eq:PDS expansion},
we recognize that the remaining terms give the probability of the event that
the last coordinate is a symmetric disjoint pivot for the pair \((A,B)\). Indeed,
\[
\begin{aligned}
&\nu_p(A_1\boxtimes B_0)
+\nu_p(A_0\boxtimes B_1)
-\nu_p(A_0\boxtimes B_0)
-\nu_p\Bigl(
(A_1\boxtimes B_0)
\cup
(A_0\boxtimes B_1)
\Bigr)
\\
&=
\nu_p\left(
\bigl((A_1\boxtimes B_0)
\cap
(A_0\boxtimes B_1)\bigr)
\setminus
(A_0\boxtimes B_0)
\right)
\,=\,
\pi_n(A,B).
\end{aligned}
\]
This concludes the proof.
\end{proof}

\noindent
For $p=1/2$, proposition~\ref{prop:diagonal-fourier} reduces to identity
 \eqref{eq:uniform-fourier-identity}.
The orthogonality of the family $\{\chi_1,\ldots,\chi_n\}$ immediately gives a first
bound on \(\|\rho(A,B)\|_2\).

\begin{corollary}
\label{cor:diagonal-bessel}
For all increasing events \(A\) and \(B\), we have
\[
pq\|\rho(A,B)\|_2
\,\leq\,
\sqrt{
D(A,B)\bigl(1-D(A,B)\bigr)
}
\,.
\]
\end{corollary}

\begin{proof}
Every \(\chi_i\) has mean zero, and therefore
\[
\int_{\mathcal E(A,B)}
\chi_i\,d\nu_p
\,=\,
\int_{\Omega\times\Omega}
\bigl(
\mathbf 1_{\mathcal E(A,B)}-D(A,B)
\bigr)
\chi_i\,d\nu_p\,.
\]
The functions \(\chi_1,\ldots,\chi_n\) form an orthonormal family.
Bessel's inequality and proposition~\ref{prop:diagonal-fourier} therefore give
\[
\begin{aligned}
p^2q^2\|\rho(A,B)\|_2^2
&\,=\,
\sum_{i=1}^n
\left(
\int_{\mathcal E(A,B)}
\chi_i\,d\nu_p
\right)^2
\,\leq\,
\int_{\Omega\times\Omega}
\bigl(
\mathbf 1_{\mathcal E(A,B)}-D(A,B)
\bigr)^2
\,d\nu_p
\\
&\,=\,
D(A,B)\bigl(1-D(A,B)\bigr)\,.
\end{aligned}
\]
Taking square roots yields the desired result.
\end{proof}

\noindent
Proposition~\ref{prop:exact-section-identity} and
proposition~\ref{prop:diagonal-fourier} provide two different perspectives on
the product space gap. The first identity gives the structure 
of the decomposition along the last coordinate and is suited to Talagrand's
induction method.
The second expresses the Fourier coefficients of $1_{\mathcal E(A,B)}$ as a function
of the influences and the symmetric disjoint pivotal probability, and is
the starting point of the proof of theorem~\ref{thm:SDP-bound}.
\section{Analytic estimates}
\label{sec:analytic-estimates}
In this section, we establish the analytic estimates needed to prove both
theorem~\ref{thm:influence-bound} and theorem~\ref{thm:SDP-bound}. 
\subsection{A biased decoupled estimate}

The estimates of this section are a generalization of
Talagrand's results in \cite{Talagrand1996} to the biased measure.
In \cite{KellerKindler2013}, Keller and Kindler extended Talagrand's work 
to higher Fourier levels, to biased measures and to non-monotone functions.
However, the results which we need for the proof of theorem~\ref{thm:influence-bound} 
are a variant of lemma~14 and theorem~16 in
\cite{KellerKindler2013}. 
In particular, the fact that 
we only consider increasing events
allows us to avoid some 
technical assumptions, and 
slightly simplifies the proof. 
Moreover, we keep track of the probability of the events
in the estimates. This slight improvement could also easily be obtained
with their methods in the increasing case.
For completeness, and because these estimates are central to the
argument, we give all the details of the proofs in the appendix.

We begin with a biased refinement of
Talagrand's lemma~3.1 in \cite{Talagrand1996}. This result is fundamental in
the study of the influences of Boolean functions, and 
is of independent interest.
Recall that the functions
\(r_1,\ldots,r_n\), defined in formula~\eqref{eq:chi_i}, form an orthonormal
family in \(L^2(\mu_p)\). 
We shall also use the biased hypercontractivity
constant \(B(p)\) introduced in \cite{KellerKindler2013}. 
\begin{lemma}
\label{lem:biased-rare-talagrand}
Let \(I\) and \(J\) form a partition of \(\unn\), and let \(A\) be an
increasing event. For a positive real number \(s\), define
\[
L_s(A)
\,=\,
\left\{\, k\in J :
\left(
\sum_{i\in I}\biggl(\int_{\Piv(k,A)} r_i\,d\mu_p\,\biggr)^2
\right)^{1/2}
\,>\,
s\,\Inf_k(A)
\,\right\}\,.
\]
Then
\[
pq
\sum_{k\in L_s(A)}
\Inf_k(A)^2
\,\leq\,
e\,\mu_p(A)
\exp\left(
-\frac{s^2}{8B(p)e}
\right)\,.
\]
\end{lemma}
\noindent
From this lemma, we derive the biased decoupled estimate. For two increasing events
\(A\) and \(B\), denote their total influence by
\[
\mathcal I(A,B)
\,=\,
\sum_{k=1}^n
\Inf_k(A)\Inf_k(B)
\,.
\]
\begin{theorem}
\label{thm:rare-biased-decoupled}
For any two increasing events \(A\) and \(B\) on the discrete cube,
\[
\sum_{k=1}^n\sum_{i\neq k}
\left|
\left(\int_{\Piv(k,A)} r_i\,d\mu_p\right)
\left(\int_{\Piv(k,B)} r_i\,d\mu_p\right)
\right|
\\
\leq
32eB(p)\,\mathcal I(A,B)
\log\left(
\frac{
4e^{3/2}\sqrt{\mu_p(A)\mu_p(B)}
}{
pq\,\mathcal I(A,B)
}
\right).
\]
We use the convention that the right-hand side is zero when \(\mathcal I(A,B)=0\).
\end{theorem}

\noindent
We shall also use the following elementary properties of the function
\(\varphi\) which appears in theorem~\ref{thm:influence-bound}.
Recall that \(\varphi(x) = x/\log(e/x))\).

\begin{lemma}
\label{lem:phi}
Let
\[
F(R,s)
=
R\varphi\left(\frac{s}{R}\right)
=
\frac{s}{\log(eR/s)},
\qquad
0\leq s\leq R\leq 1.
\]
Then the following properties hold:
\begin{enumerate}
    \item[(i)] The map \(F\) is jointly convex on the domain
    \(\{(R,s):0\leq s\leq R\leq 1\}\).

    \item[(ii)] For every \(R\in[0,1]\), we have \(F(R,0)=0\).

    \item[(iii)] The map \(F\) is non-increasing in its first variable and
    non-decreasing in its second variable.

    \item[(iv)] For every fixed \(R>0\), the map
    \[
    u
    \longmapsto
    R\varphi\left(\frac{pq\,u}{R}\right)
    \]
    is increasing and convex on \(0\leq u\leq R/pq\).

    \item[(v)] For every \(R>0\) and every $u,v$ such that 
    \(0\leq u<v\leq R/pq\),
    we have
    \[
    R\varphi\left(\frac{pq\,v}{R}\right)
    \leq
    R\varphi\left(\frac{pq\,u}{R}\right)
    +
    \frac{
    2pq\,(v-u)
    }{
    \log\left(eR/(pq\,v)\right)
    }.
    \]
\end{enumerate}
\end{lemma}
\noindent
The proof is purely technical, and is given in the appendix.
\subsection{An estimate for diagonal coefficients}

We now adapt proposition~2.2 of Talagrand
\cite{Talagrand1996} to the biased product cube.
We apply his argument to the variables \(\chi_1,\ldots,\chi_n\),
which are independent, bounded and
centered under \(\nu_p\).
% Define the constant
% \[
% M_p
% \,\coloneqq\, 
% \lVert \chi_i \rVert_\infty
% \,=\,
% \max\left(\frac{p}{q},\frac{q}{p}\right)
% \]

\begin{proposition}
\label{prop:first-level}
For every event \(E\) on the product cube,
\[
\left(
\sum_{i=1}^n
\left(
\int_E \chi_i\,d\nu_p
\right)^2
\right)^{1/2}
\,\leq\,
3\max\left(\frac{p}{q},\frac{q}{p}\right)
\nu_p(E)
\sqrt{
\log\left(
\frac{2}{\nu_p(E)}
\right)
}\,.
\]
\end{proposition}

\begin{proof}
We follow the argument of Talagrand \cite{Talagrand1996}.
The result is immediate when \(\nu_p(E)=0\). We therefore assume that
\(\nu_p(E)>0\). For every coordinate \(i\), define 
\[
c_i
\,=\,
\int_E
\chi_i\,d\nu_p
\]
If \(\sum_{i=1}^n c_i^2 = 0\),
then the left-hand side of the desired estimate vanishes, and there is
nothing to prove. We may therefore assume that \(\sum_{i=1}^n c_i^2 > 0\).
Define the normalized coefficients
\[
\alpha_i
\,=\,
\frac{c_i}{
\left(
\sum_{j=1}^n c_j^2
\right)^{1/2}
}
\,,
\]
and set \(G = \sum_{i=1}^n \alpha_i\chi_i\).
We then have
\begin{equation}
\label{eq:normalizeed-coefficients-first-level}
\sum_{i=1}^n \alpha_i^2
\,=\,
1
\,,
\qquad
\int_E G\,d\nu_p
\,=\,
\sum_{i=1}^n \alpha_i c_i
\,=\,
\left(
\sum_{i=1}^n c_i^2
\right)^{1/2}
\,.
\end{equation}
The random variables \(\chi_1,\ldots,\chi_n\) are independent and
centered, and each of them is bounded in absolute value
by 
\begin{equation}
\label{eq:def_MP}
M_p \, \coloneqq \, \max\left(\frac{p}{q},\frac{q}{p}\right)\,.
\end{equation}
Let us fix $\lambda\in\mathbb R$. Hoeffding's lemma gives that
\begin{equation}
\label{eq:hoeffding}
\mathbb E_{\nu_p}\!\left[
\exp\left(\lambda\alpha_i\chi_i\right)
\right]
\,\leq\,
\exp\left(
\frac{\lambda^2\alpha_i^2M_p^2}{2}
\right)
\,.
\end{equation}
Using independence, the fact that \(\sum_{i=1}^n\alpha_i^2=1\) and
taking the product over all \(i\) in \eqref{eq:hoeffding}, we obtain
\begin{equation}
\label{eq:sum-hoeffding}
\mathbb E_{\nu_p}\!\left[e^{\lambda G}\right]
\,\leq\,
\exp\left(
\frac{\lambda^2M_p^2}{2}
\right)
\,.
\end{equation}
The exponential Markov inequality therefore yields
\[
\forall 
r>0,
\forall 
\lambda>0,
\quad
\nu_p(G>r)
\,\leq\,
\exp\left(
\frac{\lambda^2M_p^2}{2}-\lambda r
\right)
\,.
\]
Optimizing over \(\lambda\), and applying the same estimate to \(-G\),
gives
\[
\nu_p\bigl(|G|>r\bigr)
\,\leq\,
2\exp\left(
-\frac{r^2}{2M_p^2}
\right)
\,.
\]
We choose the truncation level
\(r_0=M_p\sqrt{2\log\left(\frac{2}{\nu_p(E)}\right)}.\)
By Tonelli's theorem,
\begin{equation}
\int_E |G|\,d\nu_p
\,=\,
\int_0^\infty
\nu_p\bigl(E\cap\{\,|G|>r\,\}\bigr)\,dr
\,.
\end{equation}
Splitting the integral at \(r_0\), we obtain
\begin{equation}
\label{eq:truncation}
\begin{aligned}
\int_E |G|\,d\nu_p
&\,\leq\,
r_0\,\nu_p(E)
+
\int_{r_0}^\infty
\nu_p\bigl(|G|>r\bigr)\,dr
\\
&\,\leq\,
r_0\,\nu_p(E)
+
2\int_{r_0}^\infty
\exp\left(
-\frac{r^2}{2M_p^2}
\right)\,dr
\,.
\end{aligned}
\end{equation}
The standard Gaussian tail estimate gives
\[
\int_{r_0}^{\infty}
\exp\left(
-\frac{r^2}{2M_p^2}
\right)\,dr
\,\leq\,
\frac{M_p^2}{r_0}
\exp\left(
-\frac{r_0^2}{2M_p^2}
\right)
\,.
\]
Using the definition of \(r_0\), the inequality 
\eqref{eq:truncation} then becomes
\begin{equation}
\label{eq:G-bound}
\begin{aligned}
\int_E|G|\,d\nu_p
\,\leq\,&
M_p\nu_p(E)
\left[
\sqrt{
2\log\left(
\frac{2}{\nu_p(E)}
\right)
}
+
\frac{1}{
\sqrt{
2\log\left(
2/\nu_p(E)
\right)
}
}
\right]
\,.
\\
&\leq
3M_p\nu_p(E)
\sqrt{
\log
\left(
\frac{2}{\nu_p(E)}
\right)}
\end{aligned}
\end{equation}
The desired result follows from
\eqref{eq:normalizeed-coefficients-first-level}
and \eqref{eq:G-bound}.
\end{proof}

\noindent
Finally, we present a reformulation of proposition~\ref{prop:first-level},
which will be the main tool in the proof of theorem~\ref{thm:SDP-bound}.
\begin{lemma}
\label{lem:first-level-inversion}
Let \(M \geq 1\), \(0<z\leq R \leq 1\), and let \(0<d \leq R^2\). Assume that
\[
z
\,\leq\,
3Md
\sqrt{
\log\left(
\frac{2}{d}
\right)
}
\,.
\]
Then,
\[
d
\,\geq\,
\frac{z}{
12M
\sqrt{
\log\left(
\dfrac{6MR}{z}
\right)
}
}
\,.
\]
\end{lemma}
\begin{proof}
The proof proceeds in two steps. We first invert the estimate using the
monotonicity of the function appearing on its right-hand side, which gives
a preliminary lower bound. We then use the additional upper bound on \(d\)
to introduce the parameter \(R\) into the logarithm. 

Consider the increasing function
\(f(t) = t\sqrt{\log(2/t)}\)
on the interval \(]0,1]\). The assumption of the lemma can be written as
\begin{equation}
\label{eq:assumption}
\frac{z}{3M}
\,\leq\,
f(d)
\,.
\end{equation}
Set
\(L = \log\left(6M/z\right)\).
By evaluating \(f\) at the point
\[
t_*
\,=\,
\frac{z}{6M\sqrt L},
\]
we obtain
\[
\begin{aligned}
f(t_*)
&\,=\,
\frac{z}{6M\sqrt L}
\sqrt{
L+\log\left(2\sqrt L\right)
}
\,\leq\,
\frac{z}{3M}
\,,
\end{aligned}
\]
where we used
\(\log\left(2\sqrt L\right) \leq L\).
Thus, the assumption \eqref{eq:assumption} implies
\[
f(t_*)
\,\leq\,
\frac{z}{3M} 
\,\leq\, f(d),
\]
and the monotonicity of \(f\) gives the preliminary inequality
\begin{equation}
\label{eq:preliminary-lower-bound}
d
\,\geq\,
t_*
\,=\,
\frac{z}{6M\sqrt L}
\,.
\end{equation}
We now use the upper bound on \(d\) to introduce the parameter \(R\)
into the logarithm. Since \(d \leq R^2\), \eqref{eq:preliminary-lower-bound}
implies \(R^2 \geq z/6M\sqrt L\).
Taking logarithms gives
\[
\begin{aligned}
2\log\left(
R
\right)
&\,\geq\,
-L-\frac{1}{2}\log L
\,\geq\,
-\frac{3}{2}L
\,.
\end{aligned}
\]
Consequently,
\begin{equation}
\label{eq:l-lower-bound}
\log\left(
\frac{6MR}{z}
\right)
\,=\,
L+\log\left(
R
\right)
\,\geq\,
\frac{1}{4}L
\,.
\end{equation}
Combining \eqref{eq:preliminary-lower-bound} 
and \eqref{eq:l-lower-bound} finally yields
\[
d
\,\geq\,
\frac{z}{6M\sqrt L}
\,\geq\,
\frac{z}{
12M
\sqrt{
\log\left(
\frac{6MR}{z}
\right)
}
}
\,,
\]
as claimed.
\end{proof}

\section{Proof of the main results}
\label{sec:proofs}
The proof of theorem~\ref{thm:influence-bound} follows Talagrand's 
method of induction on the dimension, and is very similar to
his proof of the FKG case in \cite{Talagrand1996}.
Theorem~\ref{thm:SDP-bound} could also be 
proven by induction on the dimension, with analogous estimates
for the symmetric disjoint pivotal probabilities. However,
there is a more direct proof, using proposition~\ref{prop:first-level}.
In both cases, we prove the bound on \(D(A,B)\), the results announced in the 
introduction then follow directly from the inequality
\(\Delta_{BK}(A,B)\geq D(A,B)\).
\subsection{The influence lower bound}

The proof of theorem~\ref{thm:influence-bound} follows 
the sketch of proof we have given in the uniform case in
section \ref{sec:sketch}. 
The starting point is the decomposition of the gap along 
the last coordinate given by proposition~\ref{prop:exact-section-identity}.
After using the induction hypothesis on lower-dimensional gaps, 
we will use theorem~\ref{thm:rare-biased-decoupled} to control the error terms.
\begin{proof}[Proof of theorem~\ref{thm:influence-bound}]
We argue by induction on the dimension. For \(n=1\), the claim is
immediate if one of the two events is trivial.
Otherwise, necessarily \(A=B=\{1\}\),
so that \(D(A,B)=p^2\), \(\mu_p(A)=\mu_p(B)=p\),
and \(\mathcal I(A,B)=1\). Since \(B(p)\geq 1\),
\(\phi(q)\leq q\), and \(p\wedge q\leq p\),
the right-hand side of the claimed inequality is at most
\[
c\,\frac{(p\wedge q)^2}{pqB(p)}\,p\,\phi(q)
\leq c(p\wedge q)^2
\leq cp^2,
\]
so the result follows for \(c\leq 1\).

Let $n\geq 2$ and let $A,B\subseteq \{0,1\}^n$ be two increasing events.
Assume that the result in dimension $n-1$ holds.
We may assume that \(\mathcal I(A,B)>0\), 
since otherwise the result is immediate.
We introduce the notation
\[
% I
% \,=\,
% \mathcal I(A,B)
% \,,
% \qquad
R
\,=\,
\sqrt{\mu_p(A)\mu_p(B)}
\,.
\]
We start by choosing an appropriate coordinate along which we will
perform the decomposition of the gap.
The triangle inequality and
theorem~\ref{thm:rare-biased-decoupled} give
\[
\sum_{k=1}^n\sum_{i\neq k}
\left|
\left(\int_{\Piv(k,A)} r_i\,d\mu_p\right)
\left(\int_{\Piv(k,B)} r_i\,d\mu_p\right)
\right|
\\
\leq
32eB(p)\,\mathcal I(A,B)
\log\left(
\frac{
4e^{3/2}\sqrt{\mu_p(A)\mu_p(B)}
}{
pq\,\mathcal I(A,B)
}
\right).
\]
Since \(\mathcal I(A,B)=\sum_{k=1}^n\Inf_k(A)\Inf_k(B)\), there
is therefore a coordinate \(k\) such that
\[
\left|
\sum_{i\neq k}
\left(\int_{\Piv(k,A)} r_i\,d\mu_p\right)
\left(\int_{\Piv(k,B)} r_i\,d\mu_p\right)
\right|
\\
\leq
32eB(p)\,
\Inf_k(A)\Inf_k(B)
\log\left(
\frac{4e^{3/2}R}{pq\mathcal I(A,B)}
\right).
\]
Without loss of generality, we rename this coordinate as \(n\). 
For every remaining coordinate,
\eqref{eq:harmonic-identity} gives
\[
\int_{\Piv(n,A)} r_i\,d\mu_p
\,=\,
\sqrt{pq}\,
\bigl(
\Inf_i(A_1)-\Inf_i(A_0)
\bigr)
\,,
\qquad
i
\,<\,
n
\,,
\]
and the same identity holds for \(B\). Consequently,
\begin{equation}
\label{eq:chosen-coordinate}
\left|
\sum_{i<n}
\bigl(
\Inf_i(A_1)-\Inf_i(A_0)
\bigr)
\bigl(
\Inf_i(B_1)-\Inf_i(B_0)
\bigr)
\right|
\\
\,\leq\,
\frac{32eB(p)}{pq}\,
\Inf_n(A)\Inf_n(B)
\log\left(
\frac{4e^{3/2}R}{pq\mathcal I(A,B)}
\right)
.
\end{equation}
Once we have chosen our coordinate, the next step is to decompose the 
gap along it using the exact section identity given by 
proposition~\ref{prop:exact-section-identity}.
Here, we only need the inequality
\begin{equation}
\label{eq:pre-section-decomposition-influences}
\begin{aligned}
D(A,B) 
\,\geq\,&
(q-p)_+ D(A_0,B_0)
+
(p\wedge q) D(A_1,B_0)
+
(p\wedge q) D(A_0,B_1)
\\
&
+
(p-q)_+ D(A_1,B_1)
+
(p \wedge q)^2\,
\Inf_n(A)\Inf_n(B).
\end{aligned}    
\end{equation}
For simplicity, we write
\[
\begin{aligned}
\lambda_{00}=(q-p)_+, \quad \lambda_{11}=(p-q)_+, \quad \lambda_{01}=\lambda_{10}=p\wedge q.
\end{aligned}
\]
With these notations, inequality \eqref{eq:pre-section-decomposition-influences} becomes
\begin{equation}
\label{eq:section-decomposition-influences}
D(A,B)
\,\geq\,
\sum_{\varepsilon,\eta\in\{0,1\}}
\lambda_{\varepsilon\eta}
D(A_\varepsilon,B_\eta)
+
(p\wedge q)^2\Inf_n(A)\Inf_n(B)  
\end{equation}
To use the induction hypothesis, we define the  weighted average
of the influences of the sections by
\[
\overline {\mathcal I}
\,=\,
\sum_{\varepsilon,\eta\in\{0,1\}}
\lambda_{\varepsilon\eta}
\mathcal I(A_\varepsilon,B_\eta)
\,.
\]
We also denote the section probabilities and their weighted average by
\[
R_{\varepsilon\eta}
\,=\,
\sqrt{
\mu_p(A_\varepsilon)\mu_p(B_\eta)
}
\,,
\qquad
\overline R
\,=\,
\sum_{\varepsilon,\eta\in\{0,1\}}
\lambda_{\varepsilon\eta}
R_{\varepsilon\eta}
\,.
\]
Applying the induction hypothesis to each section gives the first inequality
below. The second follows from the joint convexity in
lemma~\ref{lem:phi}:
\begin{equation}
\label{eq:induction-hypothesis}
\begin{aligned}
D(A,B)
&\,\geq\,
c\frac{(p\wedge q)^2}{pqB(p)}
\sum_{\varepsilon,\eta\in\{0,1\}}
\lambda_{\varepsilon\eta}R_{\varepsilon\eta}
\varphi\left(
\frac{pq\mathcal I(A_\varepsilon,B_\eta)}{R_{\varepsilon\eta}}
\right)
+(p\wedge q)^2\Inf_n(A)\Inf_n(B)
\\
&\,\geq\,
c\frac{(p\wedge q)^2}{pqB(p)}
\overline R
\varphi\left(
\frac{pq\overline {\mathcal I}}{\overline R}
\right)
+
(p\wedge q)^2\Inf_n(A)\Inf_n(B)
\,.
\end{aligned}
\end{equation}
As in the uniform case, we need to compare $\overline {\mathcal I}$ to $\mathcal I(A,B)$ 
and $\overline R$ to $R$.
Since
\(\lambda_{00} +\lambda_{01} = \lambda_{00} +\lambda_{10} = q\)
and
\(\lambda_{10}+\lambda_{11} = \lambda_{01}+\lambda_{11} = p\),
the Cauchy-Schwarz inequality gives
\[
\begin{aligned}
\overline R
&\,\leq\,
\left(
\sum_{\varepsilon,\eta\in\{0,1\}}
\lambda_{\varepsilon\eta}
\mu_p(A_\varepsilon)
\right)^{1/2}
\left(
\sum_{\varepsilon,\eta\in\{0,1\}}
\lambda_{\varepsilon\eta}
\mu_p(B_\eta)
\right)^{1/2}
\\
&\,=\,
\sqrt{
\bigl(
q\mu_p(A_0)+p\mu_p(A_1)
\bigr)
\bigl(
q\mu_p(B_0)+p\mu_p(B_1)
\bigr)
}
\,=\,
R
\,.
\end{aligned}
\]
Moreover, lemma~\ref{lem:phi} states that
the map \((R,s)\mapsto R\varphi(s/R)\) is non-increasing in its first
variable, therefore
\[
\overline R
\varphi\left(
\frac{pq\,\overline {\mathcal I}}{\overline R}
\right)
\,\geq\,
R
\varphi\left(
\frac{pq\,\overline {\mathcal I}}{R}
\right)
\,.
\]
Consequently,
\begin{equation}
\label{eq:induction-section-bound}
D(A,B)
\,\geq\,
c\frac{(p\wedge q)^2}{pqB(p)}
R
\varphi\left(
\frac{pq\,\overline {\mathcal I}}{R}
\right)
+
(p\wedge q)^2\Inf_n(A)\Inf_n(B)
\,.
\end{equation}
We now compare \(\overline {\mathcal I}\) with \(\mathcal I(A,B)\).
For every remaining coordinate, conditioning on the last
coordinate gives
\[
\Inf_i(A)
\,=\,
q\Inf_i(A_0)+p\Inf_i(A_1)
\,,
\qquad
\Inf_i(B)
\,=\,
q\Inf_i(B_0)+p\Inf_i(B_1)
\,.
\]
A direct computation gives
\begin{equation}
\label{eq:influence-computation-1}
\begin{aligned}
&\Inf_i(A)\Inf_i(B)
-
\sum_{\varepsilon,\eta\in\{0,1\}}
\lambda_{\varepsilon\eta}
\Inf_i(A_\varepsilon)\Inf_i(B_\eta)
\\
&=\,
\bigl(q^2-\lambda_{00}\bigr)
\Inf_i(A_0)\Inf_i(B_0)
+
\bigl(pq-\lambda_{01}\bigr)
\Inf_i(A_0)\Inf_i(B_1)
\\
&\,
+
\bigl(pq-\lambda_{10}\bigr)
\Inf_i(A_1)\Inf_i(B_0)
+
\bigl(p^2-\lambda_{11}\bigr)
\Inf_i(A_1)\Inf_i(B_1)
\end{aligned}
\end{equation}
Using the definition of the coefficients
\(\lambda_{\varepsilon\eta}\), the four coefficients above are respectively
\[
(p\wedge q)^2,\quad
-(p\wedge q)^2,\quad
-(p\wedge q)^2,\quad
(p\wedge q)^2,
\]
and therefore, by using the same factorization as in \eqref{eq:influence-factorization}, equation \eqref{eq:influence-computation-1} can be rewritten as
\begin{equation}
\label{eq:influence-computation-2}
\begin{aligned}
&\Inf_i(A)\Inf_i(B)
-
\sum_{\varepsilon,\eta\in\{0,1\}}
\lambda_{\varepsilon\eta}
\Inf_i(A_\varepsilon)\Inf_i(B_\eta)
\\
&=\,
(p\wedge q)^2
\bigl(
\Inf_i(A_1)-\Inf_i(A_0)
\bigr)
\bigl(
\Inf_i(B_1)-\Inf_i(B_0)
\bigr)
\end{aligned}
\end{equation}
Summing over the remaining coordinates, and then adding $\Inf_n(A)\Inf_n(B)$,
yields
\begin{equation}
\label{eq:section-influence-comparison}
\mathcal I(A,B)
\,=\,
\overline {\mathcal I}
+
\Inf_n(A)\Inf_n(B)
+
(p\wedge q)^2
\sum_{i<n}
\bigl(
\Inf_i(A_1)-\Inf_i(A_0)
\bigr)
\bigl(
\Inf_i(B_1)-\Inf_i(B_0)
\bigr)
\,.
\end{equation}
If \(\mathcal I(A,B) \leq \overline {\mathcal I}\),
the monotonicity in the second variable of $\varphi$ and
\eqref{eq:induction-section-bound} immediately give the desired estimate.
We may therefore assume that \(\mathcal I(A,B) > \overline {\mathcal I}\).
Lemma~\ref{lem:phi} then gives
\begin{equation}
\label{eq:induction-comparison}
R
\varphi\left(
\frac{pq\mathcal I(A,B)}{R}
\right)
\,\leq\,
R
\varphi\left(
\frac{pq\,\overline {\mathcal I}}{R}
\right)
+
\frac{
2pq\,
\bigl(
\mathcal I(A,B)-\overline {\mathcal I}
\bigr)
}{
\log\left(
eR/(pq\mathcal I(A,B))
\right)
}
\,.
\end{equation}
Combining \eqref{eq:section-influence-comparison} with
\eqref{eq:chosen-coordinate}, we obtain
\begin{equation}
\label{eq:influence-comparison}
\mathcal I(A,B)-\overline {\mathcal I}
 \,\leq\,
\Inf_n(A)\Inf_n(B)
+
\frac{32e(p\wedge q)^2B(p)}{pq}\,
\Inf_n(A)\Inf_n(B)
\log\left(
\frac{4e^{3/2}R}{pq\mathcal I(A,B)}
\right)
\,.
\end{equation}
Moreover,
\[
pq\sum_{i=1}^n
\Inf_i(A)\Inf_i(B)
\,=\,
\sum_{i=1}^n
\Bigl(
\int_{A} r_i\,d\mu_p\Bigr)
\Bigl(
\int_{B} r_i\,d\mu_p
\Bigr)
\]
Using the Cauchy-Schwarz inequality and then the Bessel inequality, we obtain that 
\[
\begin{aligned}
pq\sum_{i=1}^n
\Inf_i(A)\Inf_i(B)
&\leq
\sqrt{
\sum_{i=1}^n
\left(
\int_{A} r_i\,d\mu_p\right)^2
}
\sqrt{
\sum_{i=1}^n
\left(
\int_{B} r_i\,d\mu_p
\right)^2
}
\\
&\leq
\sqrt{
\mu_p(A)\mu_p(B)}
\end{aligned}
\]
In other words, this means that \(pq\mathcal I(A,B)\leq R\). Hence,
\[
\frac{
\log\left(
4e^{3/2}R/(pq\mathcal I(A,B))
\right)
}{
\log\left(
eR/(pq\mathcal I(A,B))
\right)
}
\,\leq\,
3
\,.
\]
Moreover,
\((p\wedge q)^2 \leq pq\,\) and \(B(p) \geq 1\).
Hence, plugging \eqref{eq:influence-comparison} into \eqref{eq:induction-comparison} gives
\begin{equation}
\label{eq:induction-comparison-2}
R
\varphi\left(
\frac{pq\mathcal I(A,B)}{R}
\right)
\,\leq\,
R
\varphi\left(
\frac{pq\,\overline {\mathcal I}}{R}
\right)
+
(2+192e)\,
pqB(p)\,
\Inf_n(A)\Inf_n(B)
\,.
\end{equation}
We choose the universal constant \(c\) so that
\(c(2+192e) \leq 1\).
The inequalities \eqref{eq:induction-section-bound} 
and \eqref{eq:induction-comparison-2} 
yield the inequality on \(D(A,B)\) in dimension $n$, and the 
induction is completed.
% we obtain
% \[
% \begin{aligned}
% c\frac{(p\wedge q)^2}{pqB(p)}
% R
% \varphi\left(
% \frac{pq\mathcal I(A,B)}{R}
% \right)
% &\leq
% c\frac{(p\wedge q)^2}{pqB(p)}
% R
% \varphi\left(
% \frac{pq\,\overline{\mathcal I}}{R}
% \right)
% +
% (p\wedge q)^2\Inf_n(A)\Inf_n(B)
% \\
% &\,\leq\,
% D(A,B)
% \,,
% \end{aligned}
% \]
% where the last inequality follows from
% \eqref{eq:induction-section-bound}. This closes the induction.
\end{proof}

\subsection{The symmetric disjoint pivot and influence lower bound}
We now prove theorem~\ref{thm:SDP-bound}. We apply directly 
proposition~\ref{prop:first-level} to the defect event $\mathcal E(A,B)$. 
Rewriting the result using lemma~\ref{lem:first-level-inversion} will give 
the desired bound.
\begin{proof}[Proof of theorem~\ref{thm:SDP-bound}]
  
There is nothing to prove when \(\|\rho(A,B)\|_2=0\). We may therefore assume that
\(\|\rho(A,B)\|_2>0\).
By proposition~\ref{prop:diagonal-fourier}, the diagonal Fourier
coefficients of the defect event satisfy
\[
\forall i \in \{1,\ldots,n\},
\quad
\int_{\mathcal E(A,B)}\chi_i\,d\nu_p
=
pq\,\rho_i(A,B).
\]
Consequently,
\[
\left(
\sum_{i=1}^n
\left(
\int_{\mathcal E(A,B)}\chi_i\,d\nu_p
\right)^2
\right)^{1/2}
=
pq\|\rho(A,B)\|_2.
\]
Applying proposition~\ref{prop:first-level} to \(\mathcal E(A,B)\), 
whose probability is
\(D(A,B)\), gives
\[
pq\|\rho(A,B)\|_2
\leq
3M_pD(A,B)
\sqrt{
\log\left(
\frac{2}{D(A,B)}
\right)
},
\]
where we recall that we defined \(M_p=\max\left(p/q,q/p\right)\). 
By corollary \ref{cor:diagonal-bessel},
\[
pq\|\rho(A,B)\|_2
=
\left(
\sum_{i=1}^n
\left(
\int_{\mathcal E(A,B)}\chi_i\,d\nu_p
\right)^2
\right)^{1/2}
\leq
\sqrt{D(A,B)\bigl(1-D(A,B)\bigr)}.
\]
Since
\(D(A,B)\leq\mu_p(A)\mu_p(B)\),
it follows that
\(pq\|\rho(A,B)\|_2\leq\sqrt{\mu_p(A)\mu_p(B)}\).
We may therefore apply lemma~\ref{lem:first-level-inversion} with
\[
M=M_p,
\qquad
z=pq\|\rho(A,B)\|_2,
\qquad
d=D(A,B),
\qquad
R=\sqrt{\mu_p(A)\mu_p(B)}.
\]
We obtain
\[
z
\leq
3Md
\sqrt{
\log\left(
\frac{2}{d}
\right)
}.
\]
Since $M_p\geq 1$, $0<z\leq R\leq 1$, $0<d\leq R^2$, we can apply
lemma~\ref{lem:first-level-inversion}, which yields
\[
d
\geq
\frac{
z
}{
12M
\sqrt{
\log\left(
\frac{
6MR
}{
z
}
\right)
}
}.
\]
Substituting
\[
Q=\|\rho(A,B)\|_2,
\qquad
R=\sqrt{\mu_p(A)\mu_p(B)},
\]
\[
\frac{pq}{M_p}=(p\wedge q)^2,
\]
we obtain
\[
D(A,B)
\geq
\frac{
(p\wedge q)^2\|\rho(A,B)\|_2
}{
12
\sqrt{
\log\left(
\frac{
6\sqrt{\mu_p(A)\mu_p(B)}
}{
(p\wedge q)^2\|\rho(A,B)\|_2
}
\right)
}
},
\]
which proves the result.
\end{proof}

\section{Examples and sharpness}
\label{sec:examples}
We conclude with examples using a classical family of increasing 
events: thresholds. In one case, the bound of theorem~\ref{thm:SDP-bound} is sharp, and the symmetric 
disjoint pivot dominates the contribution. In the other case,
the bound of theorem~\ref{thm:influence-bound} is sharp.
We also exhibit an example for which neither of the estimates has the correct order.

Throughout this section, the bias \(p\) is fixed. Implicit constants may
depend on \(p\), and on any other fixed parameter appearing in the examples,
but not on the dimension.
Let \(n\geq 1\), and let \(1\leq k \leq n\). The symmetric threshold associated 
to \(k,n\) is the event that the
configuration contains at least \(k\) open coordinates. More precisely,
\[
H_{n,k}
\,=\,
\left\{\,
x\in\{0,1\}^n:
\sum_{i=1}^n x_i
\,\geq\,
k
\,\right\}
\,.
\]
For two configurations $x$ and $y$, we denote their supports
by \(X\) and \(Y\). A pair $(x,y)$ therefore admits disjoint testimony
precisely when each support has cardinality at least \(k\), and their union
has cardinality at least \(2k\). Equivalently,
\begin{equation}
\label{eq:symmetric-product-occurrence}
(x,y)\in H_{n,k}\boxtimes H_{n,k}
\quad\Longleftrightarrow\quad
\begin{cases}
|X|
\,\geq\,
k\,,\\
|Y|
\,\geq\,
k\,,\\
|X\cup Y|
\,\geq\,
2k\,.
\end{cases}
\end{equation}

\noindent
Fix a coordinate \(i\), and denote by \(X_i\) and \(Y_i\) the supports
restricted to the remaining coordinates. The same characterization shows
that \(i\) is a symmetric disjoint pivot precisely when both restricted
supports contain at least \(k\) coordinates, while their union
has exactly \(2k-1\) coordinates. More
precisely, the coordinate $i$ is a symmetric disjoint pivot in the pair $(x,y)$ 
if and only if
\begin{equation}
\label{eq:symmetric-SDP}
|X_i|
\,\geq\,
k\,,
\qquad
|Y_i|
\,\geq\,
k\,,
\qquad
|X_i\cup Y_i|
\,=\,
2k-1
\,.
\end{equation}
The coordinate \(i\) is pivotal for the event \(H_{n,k}\) in a configuration $x$
precisely when there are exactly $k-1$ other open coordinates. In other words,
$i$ is pivotal if and only if 
\begin{equation}
\label{eq:symmetric-influence}
|X_i|
\,=\,
k-1
\,.
\end{equation}
Recall that, for two quantities $f$ and $g$ depending on $n$, we 
write $f\asymp g$ 
if there exists constants $0<c<C<\infty$,  independent of n, such that
$cg<f<Cg$ for all $n$ large enough. When these constants are allowed to depend on
$p$, we write $f\asymp_p g$.
\paragraph{The OR threshold.}

Consider the threshold of level one $H_{n,1}$, which is the OR event. Its probability is
\(\mu_p(H_{n,1}) = 1-q^n\).
Among pairs in \(H_{n,1}\times H_{n,1}\), disjoint testimony fails exactly
when the two supports are equal to the same singleton. There are \(n\)
possible singletons, and hence
\(D(H_{n,1},H_{n,1}) = np^2q^{2n-2}\).
The conditions \eqref{eq:symmetric-influence} and \eqref{eq:symmetric-SDP} 
yield, after some computation, that :
\[
\|\rho(H_{n,1},H_{n,1})\|_2
\,\asymp_p\,
\|\pi(H_{n,1},H_{n,1})\|_2
\,\asymp_p\,
n^{3/2}q^{2n}
\,.
\]
Therefore, we have
\[
\frac{
\|\rho(H_{n,1},H_{n,1})\|_2
}{
\sqrt{
\log\left(
\dfrac{
6M_p\mu_p(H_{n,1})
}{
pq\|\rho(H_{n,1},H_{n,1})\|_2
}
\right)
}
}
\,\asymp_p\,
nq^{2n}
\,\asymp_p\,
D(H_{n,1},H_{n,1})
\,.
\]
Thus, the bound of theorem~\ref{thm:SDP-bound} has the correct order, 
and the quantity $\|\rho(H_{n,1},H_{n,1})\|_2$ is dominated by 
the symmetric disjoint pivots contributions.

By contrast, the quantity appearing in
theorem~\ref{thm:influence-bound} satisfies
\[
\mu_p(H_{n,1})
\varphi\left(
\frac{
pqI(H_{n,1},H_{n,1})
}{
\mu_p(H_{n,1})
}
\right)
\,\asymp_p\,
q^{2n}
\,.
\]
The influence estimate of theorem~\ref{thm:influence-bound} is 
therefore smaller than the gap by
a factor of order \(n\).
The same example shows that the exponent \(1/2\) in the logarithm of
theorem~\ref{thm:SDP-bound} is optimal. Indeed, replacing the square root
by a power \(\gamma\) strictly smaller than \(1/2\) would give a lower bound
larger than the true defect by a factor of order
\(n^{1/2-\gamma}\).
\paragraph{The majority threshold.}
We now give a simple example showing that the influence bound 
of theorem~\ref{thm:influence-bound} can have the correct order,
while the bound of theorem~\ref{thm:SDP-bound} can be substantially weaker. We work under the uniform
measure and take \(n=2m+1\). Let
\[
H_n
=
\left\{
x\in\{0,1\}^n :
\sum_{i=1}^n x_i \geq m+1
\right\}.
\]
By symmetry,
\[
\mu(H_n)=\frac12.
\]
Moreover, \(2(m+1) > n\), hence a configuration of \(H_n\times H_n\) cannot admit
disjoint witnesses for the product event. In other words, the event 
\( H_n \boxtimes H_n \) is empty.
The charcterization \eqref{eq:symmetric-influence}, along with standard
computations, give that 
\[
\mu(H_n)\,
\varphi\left(
\frac{\frac14 \mathcal I(H_n,H_n)}{\mu(H_n)}
\right)
\asymp 1.
\]
Thus, the lower bound of theorem~\ref{thm:influence-bound} is of correct
order.

We now turn to the other bound. By the characterization of symmetric
disjoint pivots in~\eqref{eq:symmetric-SDP}, if coordinate \(i\) were a
symmetric disjoint pivot, the supports restricted to the remaining coordinates
would have to satisfy
\[
|X_i\cup Y_i|
=
2(m+1)-1
=
2m+1.
\]
This is impossible, since \(X_i\cup Y_i\) is contained in a set of cardinality
\(2m\). Consequently, the probability \(\pi_i(H_n,H_n)\) is zero
for every coordinate \(i\), and 
\[
\|\rho(H_n,H_n)\|_2
=
\sqrt n\,\Inf_i(H_n)^2
\asymp
\frac1{\sqrt n}.
\]
Since \(\mu(H_n)=1/2\), the logarithmic factor in
theorem~\ref{thm:SDP-bound} is of order \(\log n\). It follows that the
corresponding lower bound is only of order
\[
\frac{\|\rho(H_n,H_n)\|_2}
{\sqrt{\log\bigl(C/\|\rho(H_n,H_n)\|_2\bigr)}}
\asymp
\frac1{\sqrt{n\log n}}.
\]
\paragraph{A global obstruction.}

The quantities of both bounds
do not capture the full defect for every symmetric
threshold. To see this, we work under the uniform measure and
we fix a real number \(\alpha\) such that $3/8 < \alpha < 1/2$. We define
the corresponding threshold by
\(k_n = \lfloor\alpha n\rfloor\) and let \(H_n = H_{n,k_n}\). For two independent uniform configurations, the normalized cardinalities
of the two supports converge in probability to \(1/2\), while that of their
union converges in probability to \(3/4\).
Since \(\alpha<1/2\), both configurations
belong to \(H_n\) with probability tending to one. On the other hand,
\(2\alpha>3/4\), hence their union has fewer
than \(2k_n\) coordinates with probability tending to one. The
characterization \eqref{eq:symmetric-product-occurrence} therefore gives that
\(D(H_n,H_n)\) converges to $1$.

We next estimate the influences. By \eqref{eq:symmetric-influence},
the influence of a coordinate is given by
\[
\operatorname{Inf}_i(H_n)
\,=\,
\P\left(
\operatorname{Bin}\left(n-1,\frac12\right)
\,=\,
k_n-1
\right)
\,.
\]
Moreover, all coordinates have the same influence, and hence
\(\mathcal I(H_n,H_n) = n\operatorname{Inf}_i(H_n)^2\).
The relevant value lies at linear distance from the mean of the binomial
distribution, hence Hoeffding's inequality gives that
\[
\begin{aligned}
\mathcal I(H_n,H_n)
\leq
e^{-c n}\,, 
\end{aligned}
\]
for all sufficiently large \(n\) and some constant \(c > 0\).
We finally estimate the symmetric disjoint pivots. By
\eqref{eq:symmetric-SDP}, their probability is bounded above by the
probability that the union of the two restricted supports has cardinality
\(2k_n-1\). Since each coordinate belongs to this union with probability
\(3/4\), independently of the other coordinates, we obtain
\[
\pi_i(H_n,H_n)
\,\leq\,
\P\left(
\operatorname{Bin}\left(n-1,\frac34\right)
\,=\,
2k_n-1
\right)
\,.
\]
Moreover, \(\|\pi(H_n,H_n)\|_2 = \sqrt{n}\,\pi_i(H_n,H_n)\).
The relevant value again lies at linear distance from the mean so we obtain
\(\|\pi(H_n,H_n)\|_2 \leq e^{-c n}\)
for all sufficiently large \(n\),
after possibly decreasing the constant \(c\).
Since
\[
\left(
\sum_{i=1}^n
\Inf_i(H_n)^4
\right)^{1/2}
\,\leq\,
\mathcal I(H_n,H_n)\,,
\]
the triangle inequality gives
\[
\|\rho(H_n,H_n)\|_2
\,\leq\,
\mathcal I(H_n,H_n)
+
\|\pi(H_n,H_n)\|_2
\,.
\]
We therefore
have
\[
\|\rho(H_n,H_n)\|_2
\,\leq\,
2e^{-c n}
\]
for all sufficiently large \(n\). Thus, the gap tends to one, whereas
both quantities appearing in our lower bounds are exponentially small.
\subsubsection*{Acknowledgements}
I warmly thank Raphaël Cerf for introducing this problem to me 
and for his helpful 
comments on the successive drafts of this paper. 
\subsubsection*{AI statement}
AI (Github Copilot, ChatGPT) was only used for typesetting, language clean-up and 
minor language translations from French. 
All mathematical content was developed and written by the author. 
\bibliographystyle{amsplain}
\bibliography{bk}
\appendix

\section{Appendix: Proofs}
In this appendix, we give the proofs of lemma \ref{lem:biased-rare-talagrand}, lemma \ref{lem:phi} and 
theorem \ref{thm:rare-biased-decoupled}.
We recall that these results are an extension of Talagrand's results to 
the biased case,
with the slight improvement that we keep track of the probability of the
events in the upper bound. 
These estimates can be considered as easier variants of the 
results of Keller and Kindler 
\cite{KellerKindler2013}, who extended Talagrand's results to
the biased case, higher Fourier levels, and to 
the case of non-monotone events. The fact that we consider 
increasing events allows us to avoid some technical assumptions, and
slightly simplifies the proofs. 

We will make use of the following large-deviation estimate, which is the
degree-one case of Keller and Kindler's theorem~10
\cite{KellerKindler2013}. If
\[
G
\,=\,
\sum_{i=1}^m c_i r_i\,,
\qquad
\sum_{i=1}^m c_i^2
\,=\,
1\,,
\]
then \(G\) has unit \(L^2\)-norm and Fourier degree one.
Hence, for every  
\(t\geq\sqrt{2B(p)e}\),
\begin{equation}
\label{eq:biased-linear-tail}
\mu_p\bigl(|G| \,\geq\, t\bigr)
\,\leq\,
\exp\left(
-\frac{t^2}{2B(p)e}
\right)\,.
\end{equation}

\begin{proof}[Proof of lemma~\ref{lem:biased-rare-talagrand}]
We follow Talagrand's strategy in
\cite[lemma~3.1]{Talagrand1996}.  We replace his dyadic decomposition by
a layer-cake formulation closer in spirit to the approach 
of Keller and Kindler in~\cite{KellerKindler2013}.
Write \(\mathcal F_I\) for the sigma-field generated by the coordinates in
\(I\). For a coordinate \(k\) in \(J\), set
\[
h_k
\,=\,
\E\left[
\1_{\Piv(k,A)}
\mid
\mathcal F_I
\right]\,.
\]
Thus, \(h_k\) is the conditional probability that coordinate \(k\) is
pivotal for \(A\) once the coordinates in \(I\) have been fixed. In
particular \(\Inf_k(A)=\E[h_k].\)
We first prove a preliminary result using the Bessel inequality.
Since \(A\) is increasing, we have the identity
\begin{equation}
\label{eq:conditional-pivotal}
\E\left[
\1_A r_k
\mid
\mathcal F_I
\right]
\,=\,
\sqrt{pq}\,h_k\,.
\end{equation}
Indeed, let \(\mathcal F_{\Omega \setminus \{k\}}\) be the sigma-field generated by
all coordinates except \(k\). Since \(A\) is increasing, conditioning on
\(\mathcal F_{\Omega \setminus \{k\}}\) gives
\[
\begin{aligned}
\mathbb E\!\left[\mathbf 1_A r_k \mid \mathcal F_{\Omega \setminus \{k\}}\right]
&\,=\,
\frac{pq}{\sqrt{pq}}
\Bigl(
\mathbf 1_A(\setcoord{x}{k}{1})
-
\mathbf 1_A(\setcoord{x}{k}{0})
\Bigr) \\
&\,=\,
\sqrt{pq}\,\mathbf 1_{\operatorname{Piv}(k,A)} \,.
\end{aligned}
\]
Since \(\mathcal F_I\subseteq\mathcal F_{\Omega \setminus \{k\}}\), taking expectations conditionally on \(\mathcal F_I\)
gives the desired identity.
Conditionally on \(\mathcal F_I\), the functions \(r_k\), with \(k\) in
\(J\), form an orthonormal family on the remaining coordinates. Hence, squaring 
\eqref{eq:conditional-pivotal}, summing over \(k\) in \(J\) and using
Bessel's inequality gives
\[
pq\sum_{k\in J}h_k^2
\,\leq\,
\E\left[
\1_A
\mid
\mathcal F_I
\right]\,.
\]
Taking expectations, we obtain
\begin{equation}
\label{eq:rare-conditional-bessel}
pq\sum_{k\in J}\E[h_k^2]
\,\leq\,
\mu_p(A)\,.
\end{equation}

\noindent
With this bound in mind, we proceed to the main argument.
Fix a coordinate \(k\) in \(L_s(A)\). There is nothing to prove when
\(\Inf_k(A)=0\). Otherwise, set
\[
c_i
\,=\,
\frac{\displaystyle \int_{\operatorname{Piv}(k,A)} r_i\,d\mu_p}
{\displaystyle
\left(
\sum_{j\in I}
\left(
\int_{\operatorname{Piv}(k,A)} r_j\,d\mu_p
\right)^2
\right)^{\!1/2}}
\qquad
G_k
\,=\,
\sum_{i\in I}c_i r_i.
\]
Then
\[
\sum_{i\in I}c_i^2
\,=\,
1,
\qquad
\sum_{i\in I}c_i\left(\int_{\operatorname{Piv}(k,A)} r_i\,d\mu_p\right)
\,=\,
\left(\sum_{i\in I}
\left(\int_{\operatorname{Piv}(k,A)}
r_i\,d\mu_p\right)
^2\right)^{1/2}.
\]
Since \(G_k\) is \(\mathcal F_I\)-measurable and
\(\int_{\operatorname{Piv}(k,A)} r_i\,d\mu_p=\E[r_i h_k]\),

\begin{equation}
\label{eq:assumption-rare-tail}
\E[G_kh_k]
\,=\,
\sum_{i\in I}\left(c_i\int_{\operatorname{Piv}(k,A)} r_i\,d\mu_p\right)
\,>\,
s\,\Inf_k(A),
\end{equation}
 because \(k\) belongs to \(L_s(A)\).
As \(h_k\) is non-negative, we have
\(\E[|G_k|h_k] > s\,\Inf_k(A)\).
The contribution of the event where \(|G_k|\) is at most \(s/2\) is
bounded by
\[
\E\left[
|G_k|h_k
\mathbf 1_{\{\, |G_k| \,\leq\, s/2 \,\}}
\right]
\,\leq\,
\frac{s}{2}\E[h_k]
\,=\,
\frac{s}{2}\Inf_k(A)\,.
\]
Consequently,
\begin{equation}
\label{eq:Gk-lower-bound}
\E\left[
|G_k|h_k
\mathbf 1_{\{\, |G_k| \,>\, s/2 \,\}}
\right]
\,\geq\,
\frac{s}{2}\Inf_k(A)\,.
\end{equation}
By Cauchy-Schwarz,
\begin{equation}
\label{eq:rare-tail-cauchy}
\frac{s^2}{4}\Inf_k(A)^2
\,\leq\,
\E\left[
G_k^2\mathbf 1_{\{\, |G_k| \,>\, s/2 \,\}}
\right]
\E\left[
h_k^2\mathbf 1_{\{\, |G_k| \,>\, s/2 \,\}}
\right]\,.
\end{equation}
Suppose first that
\(s \geq 2\sqrt{2B(p)e}\).
The identity
\[
\mathbb E[Z]
\,=\,
\int_0^\infty \mu_p(Z>t)\,dt \,,
\]
valid for every non-negative random variable \(Z\), applied to
\(Z=G_k^2\mathbf 1_{\{|G_k|>s/2\}}\), gives
\[
\mathbb E\!\left[
G_k^2\mathbf 1_{\{|G_k|>s/2\}}
\right]
\,=\,
\frac{s^2}{4}\,
\mu_p\!\left(|G_k|>\frac{s}{2}\right)
+
\int_{s^2/4}^{\infty}
\mu_p\!\left(|G_k|>\sqrt{t}\right)\,dt \,.
\]
Since \(G_k\) has degree one and unit \(L^2\)-norm, the tail
estimate~\eqref{eq:biased-linear-tail} yields
\begin{equation}
\label{eq:tail-estimate-upper-bound}
\begin{aligned}
\mathbb E\!\left[
G_k^2\mathbf 1_{\{|G_k|>s/2\}}
\right]
&\,\leq\,
\frac{s^2}{4}
\exp\!\left(-\frac{s^2}{8B(p)e}\right)
+
\int_{s^2/4}^{\infty}
\exp\!\left(-\frac{t}{2B(p)e}\right)\,dt \\
&\,=\,
\left(\frac{s^2}{4}+2B(p)e\right)
\exp\!\left(-\frac{s^2}{8B(p)e}\right)\,\leq\,
\frac{s^2}{2}
\exp\!\left(-\frac{s^2}{8B(p)e}\right),
\end{aligned}
\end{equation}
where the last inequality follows from
\(s^2\geq 8B(p)e\). Combining \eqref{eq:tail-estimate-upper-bound} with 
\eqref{eq:rare-tail-cauchy} yields
\[
\Inf_k(A)^2
\,\leq\,
2
\exp\left(
-\frac{s^2}{8B(p)e}
\right)
\E\left[
h_k^2\mathbf 1_{\{\, |G_k| \,>\, s/2 \,\}}
\right]\,.
\]
Summing over the coordinates \(k\) in \(L_s\), multiplying by \(pq\), and
using \eqref{eq:rare-conditional-bessel} to bound the second factor, we obtain
\[
\begin{aligned}
pq\sum_{k\in L_s}\Inf_k(A)^2
&\,\leq\,
2
\exp\left(
-\frac{s^2}{8B(p)e}
\right)
\E\left[
pq\sum_{k\in L_s}
h_k^2\mathbf 1_{\{\, |G_k| \,>\, s/2 \,\}}
\right]
\\
&\,\leq\,
2\mu_p(A)
\exp\left(
-\frac{s^2}{8B(p)e}
\right)\,.
\end{aligned}
\]
It remains to consider
\(0 < s < 2\sqrt{2B(p)e}\).
By Jensen's inequality and \eqref{eq:rare-conditional-bessel},
\[
pq\sum_{k\in L_s}\Inf_k(A)^2
\,\leq\,
pq\sum_{k\in J}\E[h_k^2]
\,\leq\,
\mu_p(A)\,.
\]
Since
\(\exp\left(-\frac{s^2}{8B(p)e}\right) \geq e^{-1}\),
we get
\[
pq\sum_{k\in L_s}\Inf_k(A)^2
\,\leq\,
e\,\mu_p(A)
\exp\left(
-\frac{s^2}{8B(p)e}
\right)\,.
\]
Since \(2 \leq e\), this completes the proof.
\end{proof}

\begin{proof}[Proof of theorem~\ref{thm:rare-biased-decoupled}]
For simplicity, for distinct coordinates \(i\) and
\(k\), set
\[
\beta_{k,i}(A)
\,=\,
\int_{\Piv(k,A)} r_i\,d\mu_p\,.
\]
We start with the case of a fixed partition of the coordinates. The general
case will be obtained by averaging over a random partition.
Fix a partition \(U\sqcup V=\unn,\)
and denote by \(I_V\) the contribution of the coordinates in \(V\), namely
\[
I_V 
\,=\,
\sum_{j\in V}\Inf_j(A)\Inf_j(B).
\]
We first prove
\begin{equation}
\label{eq:rare-fixed-partition}
\sum_{j\in V}\sum_{i\in U}
\left|
\beta_{j,i}(A)\beta_{j,i}(B)
\right|
\,\leq\,
16eB(p)\,I_V
\log\left(
\frac{
2e^{3/2}\sqrt{\mu_p(A)\mu_p(B)}
}{
pq\,I_V
}
\right)
\,.
\end{equation}
If \(I_V = 0\), then every term on the left-hand side of
\eqref{eq:rare-fixed-partition} is zero. We may therefore assume that
\(I_V > 0\). For a positive real number \(t\), define
\[
\mathcal L_t
\,=\,
\left\{\,
j\in V:
\sum_{i\in U}
\left|
\beta_{j,i}(A)\beta_{j,i}(B)
\right|
\,>\,
t\,\Inf_j(A)\Inf_j(B)
\,\right\}
\,.
\]
The definition of \(\mathcal L_t\) gives, for every coordinate \(j\) in
\(V\),
\[
\begin{aligned}
\sum_{i\in U}
\left|
\beta_{j,i}(A)\beta_{j,i}(B)
\right|
\,=\,
\Inf_j(A)\Inf_j(B)
\int_0^\infty
\mathbf 1_{\{\,j\in\mathcal L_t\,\}}\,dt
\,.
\end{aligned}
\]
By Tonelli's theorem,
\begin{equation}
\label{eq:level-set-identity}
\begin{aligned}
\sum_{j\in V}\sum_{i\in U}
\left|
\beta_{j,i}(A)\beta_{j,i}(B)
\right|
\,=\,
\int_0^\infty
\sum_{j\in\mathcal L_t}
\Inf_j(A)\Inf_j(B)\,dt
\,.
\end{aligned}
\end{equation}
Therefore, to bound the left-hand side of \eqref{eq:rare-fixed-partition},
it suffices to control the influences of coordinates in \(\mathcal L_t\).
Define the corresponding set for the first event by
\[
\mathcal L_{\sqrt t}^{A}
\,=\,
\left\{\,
j\in V:
\left(
\sum_{i\in U}\beta_{j,i}(A)^2
\right)^{1/2}
\,>\,
\sqrt t\,\Inf_j(A)
\,\right\}
\,,
\]
and define \(\mathcal L_{\sqrt t}^{B}\) analogously. The
Cauchy-Schwarz inequality gives
\[
\mathcal L_t
\,\subseteq\,
\mathcal L_{\sqrt t}^{A}
\cup
\mathcal L_{\sqrt t}^{B}
\,.
\]
Lemma~\ref{lem:biased-rare-talagrand} gives
\begin{equation}
\label{eq:level-set-influence-bound}
pq
\sum_{j\in\mathcal L_{\sqrt t}^{A}}
\Inf_j(A)^2
\,\leq\,
e\,\mu_p(A)
\exp\left(
-\frac{t}{8B(p)e}
\right)
\,.
\end{equation}
The Bessel inequality also gives 
\begin{equation}
\label{eq:bessel-bound}
pq\sum_{j=1}^n
\Inf_j(B)^2
\,=\,
\sum_{j=1}^n
\left(
\int_{B}
r_j\, d\mu_p
\right)^2
\leq
\mu_p(B)
\end{equation}
Using Cauchy-Schwarz and bounding the result with
\eqref{eq:level-set-influence-bound} and 
\eqref{eq:bessel-bound}, we
obtain
\[
\begin{aligned}
pq
\sum_{j\in\mathcal L_{\sqrt t}^{A}}
\Inf_j(A)\Inf_j(B)
&\,\leq\,
\left(
pq\sum_{j\in\mathcal L_{\sqrt t}^{A}}
\Inf_j(A)^2
\right)^{1/2}
\left(
pq\sum_{j=1}^n
\Inf_j(B)^2
\right)^{1/2}
\\
&\,\leq\,
\sqrt e\,\sqrt{\mu_p(A)\mu_p(B)}
\exp\left(
-\frac{t}{16B(p)e}
\right)
\,.
\end{aligned}
\]
The same estimate holds after interchanging the two events. Consequently,
\begin{equation}
\label{eq:rare-product-level-set}
pq
\sum_{j\in\mathcal L_t}
\Inf_j(A)\Inf_j(B)
\,\leq\,
2\sqrt e\,\sqrt{\mu_p(A)\mu_p(B)}
\exp\left(
-\frac{t}{16B(p)e}
\right)
\,.
\end{equation}
We now return to identity \eqref{eq:level-set-identity}. 
Choose the splitting point
\[
t_0
\,=\,
16B(p)e
\log\left(
\frac{
2\sqrt e\,\sqrt{\mu_p(A)\mu_p(B)}
}{
pq\,I_V
}
\right)
\,.
\]
The Cauchy-Schwarz inequality and Bessel bounds give
\begin{equation*}
\begin{aligned}
pq\,I_V
&\,\leq\,
\left(
pq\sum_{j\in V}\Inf_j(A)^2
\right)^{1/2}
\left(
pq\sum_{j\in V}\Inf_j(B)^2
\right)^{1/2}
\\
&\,\leq\,
\sqrt{\mu_p(A)\mu_p(B)}
\,.
\end{aligned}
\end{equation*}
This ensures that the splitting point \(t_0\) is positive. Splitting
the integral at \(t_0\) in \eqref{eq:level-set-identity} and using
\eqref{eq:rare-product-level-set}, we obtain
\[
\sum_{j\in V}\sum_{i\in U}
\left|
\beta_{j,i}(A)\beta_{j,i}(B)
\right|
\,\leq\,
t_0I_V
+
\frac{
2\sqrt e\,\sqrt{\mu_p(A)\mu_p(B)}
}{
pq
}
\int_{t_0}^{\infty}
\exp\left(
-\frac{t}{16B(p)e}
\right)\,dt
\]
Evaluating the integral and using the definition of \(t_0\),
this becomes
\[
\begin{aligned}
\sum_{j\in V}\sum_{i\in U}
\left|
\beta_{j,i}(A)\beta_{j,i}(B)
\right|
&\,\leq\,
t_0I_V
+
\frac{
32B(p)e^{3/2}\sqrt{\mu_p(A)\mu_p(B)}
}{
pq
}
\exp\left(
-\frac{t_0}{16B(p)e}
\right)
\\
&\,=\,
t_0I_V
+
16B(p)e\,I_V
\\
&\,=\,
16B(p)e\,I_V
\left[
\log\left(
\frac{
2\sqrt e\,\sqrt{\mu_p(A)\mu_p(B)}
}{
pq\,I_V
}
\right)
+1
\right]
\\
&\,=\,
16B(p)e\,I_V
\log\left(
\frac{
2e^{3/2}\sqrt{\mu_p(A)\mu_p(B)}
}{
pq\,I_V
}
\right)
\,.
\end{aligned}
\]
This proves \eqref{eq:rare-fixed-partition}.

It remains to average over the partitions. Choose the partition at random
by placing each coordinate in \(V\) independently with probability
\(1/2\), and let \(U\) be its complement. Each ordered pair of distinct
coordinates appears in the fixed-partition sum with probability \(1/4\).
Moreover, \(\mathbb E[I_V]=\frac{1}{2}\mathcal I(A,B)\).
The function defined by
\[
x
\longmapsto
x\log\left(
\frac{
2e^{3/2}\sqrt{\mu_p(A)\mu_p(B)}
}{
pq\,x
}
\right)
\]
is concave on the relevant interval and has value zero at the origin.
Hence, \eqref{eq:rare-fixed-partition} and Jensen's inequality give
\[
\frac{1}{4}
\sum_{k=1}^n\sum_{i\neq k}
\left|
\beta_{k,i}(A)\beta_{k,i}(B)
\right|
\,\leq\,
8eB(p)\,\mathcal I(A,B)
\log\left(
\frac{
4e^{3/2}\sqrt{\mu_p(A)\mu_p(B)}
}{
pq\,\mathcal I(A,B)
}
\right)
\,.
\] 
Multiplying by \(4\) proves the theorem.

\end{proof}

\begin{proof}[Proof of lemma~\ref{lem:phi}]
Introduce the auxiliary function \(L\) by
\(L(t) = \log(e/t)\).
For every positive real number \(t\) at most \(1\),
\[
\varphi''(t)
\,=\,
\frac{L(t)+2}{tL(t)^3}
\,>\,
0
\,.
\]
Thus, \(\varphi\) is convex, and 
\[
(R,s)
\longmapsto
R\varphi\left(\frac{s}{R}\right)
\]
is jointly convex. Direct differentiation gives
\[
\begin{aligned}
\frac{\partial}{\partial R}
\left[
R\varphi\left(\frac{s}{R}\right)
\right]
&\,=\,
-\frac{s}{R\log(eR/s)^2}
\,\leq\,
0
\,,
\\
\frac{\partial}{\partial s}
\left[
R\varphi\left(\frac{s}{R}\right)
\right]
&\,=\,
\frac{\log(eR/s)+1}{\log(eR/s)^2}
\,\geq\,
0
\,.
\end{aligned}
\]
Fix a positive real number \(R\), and introduce the function
\(L_R(u)=\log\left(\frac{eR}{pq\,u}\right).\)
For every positive real number \(u\) at most \(R/(pq)\), direct
differentiation gives
\[
\begin{aligned}
\frac{d}{du}
\left[
R\varphi\left(\frac{pq\,u}{R}\right)
\right]
&\,=\,
pq\,\frac{L_R(u)+1}{L_R(u)^2}
\,\leq\,
\frac{2pq}{L_R(u)}
\,,
\\
\frac{d^2}{du^2}
\left[
R\varphi\left(\frac{pq\,u}{R}\right)
\right]
&\,=\,
pq\,\frac{L_R(u)+2}{uL_R(u)^3}
\,>\,
0
\,.
\end{aligned}
\]
For every real number \(t\) between \(u\) and \(v\), monotonicity gives
\(L_R(t) \geq L_R(v)\).
Integrating the first derivative between the two endpoints proves the last
estimate. The boundary cases follow by continuity.
\end{proof}

\end{document}